\documentclass{amsart}

\usepackage{geometry}
\usepackage{latexsym,amssymb}
\usepackage[english]{babel}
\usepackage{amsfonts, amsthm, amsmath}
\usepackage{mathrsfs}

\newtheorem{prop}{Proposition}[section]
\newtheorem{lem}[prop]{Lemma}
\newtheorem{cor}[prop]{Corollary}
\newtheorem{thm}[prop]{Theorem}

\theoremstyle{definition}

\newtheorem{rem}[prop]{Remark}

\newtheorem{defi}[prop]{Definition}
\newtheorem{ex}[prop]{Example}

\newtheorem*{KN}{Crazy Knight's Tour Problem}

\def\S{\mathbb{S}}
\def\Z{\mathbb{Z}}

\newcommand{\probname}{Crazy Knight's Tour Problem}

\def\G{\Gamma}

\def\H{\mathrm{H}}
\def\lcm{\mathrm{lcm}}
\def\Cay{\mathsf{Cay}}

\def\S{\mathcal{S}}
\def\E{\mathcal{E}}
\def\D{\mathcal{D}}
\def\B{\mathcal{B}}

\def\R{\mathcal{R}}
\def\C{\mathcal{C}}

\begin{document}

\title{Relative Heffter arrays and biembeddings}

\author{S. Costa}
\address{DICATAM - Sez. Matematica, Universit\`a degli Studi di Brescia, Via
Branze 43, I-25123 Brescia, Italy}
\email{simone.costa@unibs.it}

\author{A. Pasotti}
\address{DICATAM - Sez. Matematica, Universit\`a degli Studi di Brescia, Via
Branze 43, I-25123 Brescia, Italy}
\email[Corresponding author]{anita.pasotti@unibs.it}

\author{M.A. Pellegrini}
\address{Dipartimento di Matematica e Fisica, Universit\`a Cattolica del Sacro Cuore, Via Musei 41,
I-25121 Brescia, Italy}
\email{marcoantonio.pellegrini@unicatt.it}

\begin{abstract}
Relative Heffter arrays, denoted by $\H_t(m,n; s,k)$, have been introduced as a generalization of the classical concept of Heffter array.
A $\H_t(m,n; s,k)$ is an $m\times n$ partially filled  array  with elements in $\Z_v$, where
$v=2nk+t$, whose rows contain $s$ filled cells and whose columns contain $k$ filled cells,
such that the elements in every row and column sum to zero and, for every $x\in \Z_v$ not belonging to the subgroup of
order $t$, either $x$ or $-x$ appears in the array.
In this paper we show how relative Heffter arrays can be used to construct biembeddings of
cyclic cycle decompositions of the complete multipartite graph
$K_{\frac{2nk+t}{t}\times t}$ into an orientable surface.
In particular, we construct such biembeddings providing integer globally simple  square relative Heffter arrays for
$t=k=3,5,7,9$ and $n\equiv 3 \pmod 4$ and for $k=3$ with $t=n,2n$, any odd $n$.
\end{abstract}


\maketitle

\section{Introduction}
An $m \times n$  partially filled (p.f., for short) array on a set $\Omega$ is an $m \times n$ matrix
whose elements belong to $\Omega$ and where we also allow some cells to be empty.
The following class of p.f. arrays was introduced in \cite{CMPPRelative}, generalizing the ideas
of \cite{A}:

\begin{defi}\label{def:RelativeH}
Let $v=2nk+t$ be a positive integer and let $J$ be the subgroup of $\Z_v$ of order $t$.
 A $\H_t(m,n; s,k)$ \emph{Heffter array  over $\Z_v$ relative to $J$} is an $m\times n$ p.f.  array
 with elements in $\Z_v$ such that:
\begin{itemize}
\item[($\rm{a})$] each row contains $s$ filled cells and each column contains $k$ filled cells;
\item[($\rm{b})$] for every $x\in \Z_{2nk+t}\setminus J$, either $x$ or $-x$ appears in the array;
\item[($\rm{c})$] the elements in every row and column sum to zero.
\end{itemize}
\end{defi}

 Trivial necessary conditions for the existence of a $\H_t(m,n; s,k)$ are
that $t$ divides $2nk$, $nk=ms$, $3\leq s \leq n$ and $3\leq k \leq m$.
If $\H_t(m,n; s,k)$ is a square array, it will be denoted by $\H_t(n;k)$.
A relative Heffter array is called \emph{integer} if Condition ($\rm{c}$) in Definition \ref{def:RelativeH} is
strengthened so that the elements in every row and in every
column, viewed as integers in
$\pm\left\{ 1, \ldots, \left\lfloor \frac{2nk+t}{2}\right\rfloor \right\} ,$
sum to zero in $\Z$.
We remark that, if $t=1$, namely if $J$ is the trivial subgroup of $\Z_{2nk+1}$, we find again the classical concept of a (integer) Heffter
array, see \cite{A,ABD, ADDY, BCDY, CDDY,CMPPHeffter,DM,DW}.
In particular, in \cite{CDDY} it was proved that Heffter arrays $\H_1(n;k)$ exist
for all $n\geq k\geq 3$,
while by \cite{ADDY, DW} integer Heffter arrays $\H_1(n;k)$ exist if and only if
the additional condition $nk\equiv 0,3\pmod 4$ holds.
At the moment, the only known results concerning relative Heffter arrays are described in \cite{CMPPRelative, MP}.
Some necessary conditions for the existence of an integer $\H_t(n;k)$ are given by the following.

\begin{prop}\label{prop:necc}\cite{CMPPRelative}
Suppose that there exists an integer $\H_t(n;k)$ for some $n\geq k\geq 3$ and some divisor $t$ of $2nk$.
\begin{itemize}
\item[(1)] If $t$ divides $nk$, then
$nk\equiv 0 \pmod 4$  or  $nk\equiv -t \equiv \pm 1\pmod 4.$
\item[(2)] If $t=2nk$, then $k$ must be even.
\item[(3)] If $t\neq 2nk$ does not divide $nk$, then
$t+2nk\equiv 0 \pmod 8.$
\end{itemize}
\end{prop}

We point out that these conditions are not sufficient, in fact in the same paper the authors show that there is no integer
$\H_{3n}(n;3)$ and no integer $\H_8(4;3)$.

The \emph{support} of an integer Heffter array $A$, denoted by $supp(A)$, is defined to be the set of the absolute
values of the elements
contained in $A$.
It is immediate to see that an integer $\H_2(n;k)$ is nothing but an integer $\H_1(n;k)$, since in both cases the support
is $\{1,2,\ldots,nk\}$.

In this paper we study the connection between  relative Heffter arrays and biembeddings.
In particular, in Section \ref{sec:orderings} we recall well known definitions and results about
simple orderings and cycle decompositions. Then, in Section \ref{sec:biembedding} we explain how relative
Heffter arrays $\H_t(n;k)$ can be used to construct biembeddings of cyclic $k$-cycle decompositions of the complete multipartite graph
$K_{\frac{2nk+t}{t}\times t}$ into an orientable surface.
Direct constructions of globally simple integer  $\H_t(n;3)$ with $t=n,2n$ for any odd $n$ and of
globally simple integer $\H_k(n;k)$ for $k=7,9$ and $n\equiv 3 \pmod 4$ are described in Section
\ref{sec:constructions}.
Combining the results of these sections we prove the following.
\begin{thm}\label{main}
There exists a cellular biembedding of a pair of cyclic $k$-cycle decompositions of  $K_{\frac{2nk+t}{t}\times t}$
into an orientable surface in each of the following cases:
\begin{itemize}
 \item[(1)] $k=3$, $t\in \{n,2n\}$ and $n$ is odd;
  \item[(2)] $k\in \{3,5,7,9\}$, $t=k$ and $n\equiv 3 \pmod 4$.
\end{itemize}
\end{thm}

 Finally, in Section  \ref{sec:Arch} we introduce a further generalization, called \emph{Archdeacon array}, of the classical concept
of Heffter array. We show some examples and how both cycle decompositions and biembeddings can be obtained also using these arrays.

\section{Simple orderings and cycle decompositions}\label{sec:orderings}

Given two integers $a\leq b$, we denote by $[a,b]$ the interval containing the integers
$\{a,a+1,\ldots,b\}$.
If $a>b$, then $[a,b]$ is empty.

If $A$ is an  $m\times n$ p.f.  array,
the rows and the columns of $A$ will be denoted by $\overline{R}_1,\ldots,\overline{R}_m$ and by
$\overline{C}_1,\ldots,\overline{C}_n$, respectively.
We will denote by $\E(A)$ the unordered list of the elements of the filled cells of $A$.
Analogously, by $\E(\overline{R}_i)$ and $\E(\overline{C}_j)$ we mean the unordered lists of elements  of the $i$-th 
row and of the $j$-th column, respectively, of $A$.
Also, we define the skeleton of $A$, denoted by $skel(A)$, to be the set of the filled positions of $A$.

Given  a finite subset $T$ of an abelian group $G$ and an ordering $\omega=(t_1,t_2,\ldots,t_k)$ of the elements
in $T$,  let $s_i=\sum_{j=1}^i t_j$, for any $i\in[1,k]$,
be the $i$-th partial sum of $\omega$ and  set $\S(\omega)=(s_1,\ldots,s_k)$.
The ordering $\omega$ is said to be \emph{simple} if $s_b\neq s_c$ for all $1\leq b <  c\leq k$ or,
equivalently,
if there is no proper subsequence of $\omega$ that sums to $0$.
Note that if $\omega$ is a simple ordering so is $\omega^{-1}=(t_k,t_{k-1},\ldots,t_1)$.
We point out that there are several interesting problems and conjectures about distinct partial sums:
see, for instance, \cite{AL, ADMS, CMPPSums, HOS, O}.
Given an $m \times n$ p.f. array $A$, by $\omega_{\overline{R}_i}$ and $\omega_{\overline{C}_j}$ we will
denote, respectively,  an ordering of $\E(\overline{R}_i)$ and of $\E(\overline{C}_j)$.  If for any $i\in[1, m]$
and for any $j\in[1,n]$, the orderings $\omega_{\overline{R}_i}$ and $\omega_{\overline{C}_j}$ are simple, we define
by
  $\omega_r=\omega_{\overline{R}_1}\circ \ldots \circ\omega_{\overline{R}_m}$ the simple ordering for the rows and
  by $\omega_c=\omega_{\overline{C}_1}\circ \ldots \circ\omega_{\overline{C}_n}$ the simple ordering for the columns.
Moreover, by \emph{natural ordering} of a row (column) of $A$ we mean the ordering from left to right (from top to
bottom).
A p.f. array $A$ on an abelian group $G$ is said to be
	
\begin{itemize}
\item \emph{simple} if each row and each column of $A$ admits a simple ordering;
\item \emph{globally simple} if the natural ordering of each row and each column of $A$ is simple.
\end{itemize}

Clearly if $k\leq 5$, then every square relative Heffter array is (globally) simple.

We recall some basic definitions about graphs and graph decompositions.
Given a graph $\G$, by $V(\G)$ and $E(\G)$ we mean the vertex set and the edge set of $\G$, respectively.
We will denote by $K_v$ the complete graph of order $v$
and by $K_{q\times r}$ the complete multipartite graph with $q$ parts each of size $r$.
Obviously  $K_{q\times 1}$ is nothing but the complete graph $K_q$.
Let $G$ be an additive group (not necessarily abelian) and let $\Lambda \subseteq G\setminus \{0\}$
such that $\Lambda=-\Lambda$, which means that for every $\lambda\in \Lambda$ we have also $-\lambda \in \Lambda$.
The Cayley graph on $G$ with connection set $\Lambda$, denoted by $\Cay[G:\Lambda]$, is the simple graph having $G$ as
vertex set
and such that two vertices $x$ and $y$ are adjacent if and only if $x-y\in \Lambda$.
Note that, if $\Lambda=G\setminus \{0\}$, the Cayley graph is the complete graph whose vertex set is $G$ and,
if $\Lambda=G\setminus J$ for some subgroup $J$ of $G$, the Cayley graph is the complete multipartite graph $K_{q\times r}$ where $q=|G:J|$ and $r=|J|$.

The following are well known definitions and results which can be found, for instance, in \cite{BP2007}.
Let $\G$ be a subgraph of a graph $K$.
A $\Gamma$-\emph{decomposition} of $K$ is a set $\D$
of subgraphs of $K$ isomorphic to $\G$ whose edges partition $E(K)$.
If the vertices of $K$ belong to a group $G$, given $g\in G$, by $\G+g$ one means the graph whose vertex set is
$V(\G)+g$ and whose edge set is $\{\{x+g,y+g\}\mid \{x,y\}\in E(\G)\}$.
An \emph{automorphism group} of a $\G$-decomposition $\D$ of $K$ is a group of bijections on $V(K)$
leaving $\D$ invariant.
 A $\G$-decomposition of $K$ is said to be \emph{regular under a group} $G$ or $G$-\emph{regular}
 if it admits $G$ as an automorphism group acting sharply transitively on $V(K)$.
 Here we consider cyclic cycle decompositions, namely decompositions which are regular
 under a cyclic group and with $\G$ a cycle.
Finally,
  two graph decompositions $\D$ and $\D'$ of a simple graph $K$ are said \emph{orthogonal} if and only if for any $B$
of $\D$
  and any $B'$ of $\D'$, $B$ intersects $B'$ in at most one edge.

 The relationship between simple relative Heffter arrays and cyclic cycle decompositions of the complete multipartite graph is explained in
\cite{CMPPRelative}. Here we briefly recall the following result.

 \begin{prop}\label{HeffterToDecompositions}\cite[Proposition 2.9]{CMPPRelative}
  Let $A$ be a $\H_t(m,n;s,k)$ simple with respect to the orderings $\omega_r$ and $\omega_c$. Then:
\begin{itemize}
\item[(1)] there exists a cyclic $s$-cycle decomposition $\D_{\omega_r}$ of $K_{\frac{2ms+t}{t}\times t}$;
\item[(2)] there exists a cyclic $k$-cycle decomposition $\D_{\omega_c}$ of $K_{\frac{2nk+t}{t}\times t}$;
\item[(3)] the cycle decompositions $\D_{\omega_r}$ and $\D_{\omega_c}$ are orthogonal.
\end{itemize}
\end{prop}

The arrays we are going to construct are square  with a diagonal structure, so it is convenient
to introduce the following notation.
If $A$ is an $n\times n$ array, for $i\in[1,n]$ we define the $i$-th diagonal
$$D_i=\{(i,1),(i+1,2),\ldots,(i-1,n)\}.$$
Here all the arithmetic on the row and the column indices is performed modulo $n$, where the set of reduced residues is
$\{1,2,\ldots,n\}$.
We say that the diagonals $D_i,D_{i+1},\ldots, D_{i+r}$ are \emph{consecutive diagonals}.

\begin{defi}
  Let $k\geq 1$ be an integer. We will say that a square p.f. array $A$ of size $n\geq k$ is
\begin{itemize}
 \item \emph{$k$-diagonal}   if the non empty cells of $A$ are exactly those of $k$ diagonals;
 \item \emph{cyclically $k$-diagonal}
  if the nonempty cells of $A$ are exactly those of $k$ consecutive diagonals.
\end{itemize}
\end{defi}
Let $A$ be a $k$-diagonal array of size $n> k$. A set $S=\{D_{r+1},D_{r+2},\ldots, D_{r+\ell}\}$ is said to be
an \emph{empty strip of width $\ell$} if  $D_{r+1},D_{r+2},\ldots,D_{r+\ell}$ are empty diagonals, while $D_r$ and $D_{r+\ell+1}$ are filled diagonals.
\begin{defi}
Let $A$ be a $k$-diagonal array of size $n > k$.
We will say that $A$ is a \emph{$k$-diagonal array with width $\ell$}
if all the empty strips of $A$ have width $\ell$.
\end{defi}
An array of this kind will be given in Example \ref{4.9}.

\section{Relation with biembeddings}\label{sec:biembedding}

In \cite{A}, Archdeacon introduced Heffter arrays also in view of their applications and, in particular, since they are useful for finding biembeddings of cycle decompositions, as shown, for instance, in
\cite{CDDYbiem, CMPPHeffter, DM}.
In this section, generalizing some of Archdeacon's results we show how starting from a relative
Heffter array it is possible to
obtain suitable biembeddings.

We recall the following definition, see \cite{Moh}.

\begin{defi}
An \emph{embedding} of a graph $\G$ in a surface $\Sigma$ is a continuous injective
mapping $\psi: \G \to \Sigma$, where $\G$ is viewed with the usual topology as $1$-dimensional simplicial complex.
\end{defi}

The  connected components of $\Sigma \setminus \psi(\G)$ are called $\psi$-\emph{faces}.
If each $\psi$-face is homeomorphic to an open disc, then the embedding $\psi$ is said to be \emph{cellular}.

\begin{defi}
 A \emph{biembedding} of  two cycle decompositions $\D$ and $\D'$ of a simple graph $\G$  is a face $2$-colorable embedding
  of $\G$ in which one color class is comprised of the cycles in $\D$ and
  the other class contains the cycles in $\D'$.
\end{defi}

Following the notation given in \cite{A}, for every edge $e$ of a  graph $\G$, let $e^+$ and $e^-$ denote
its two possible directions
and let $\tau$ be the involution swapping $e^+$ and $e^-$ for every $e$.
Let $D(\G)$ be the set of all directed edges of $\G$ and, for any $v\in V(\G)$, call $D_v$ the set of edges directed out
of $v$.
A local rotation $\rho_v$ is a cyclic permutation of $D_v$. If we select a local rotation for each vertex of $\G$, then
all together
they form a rotation of $D(\G)$. We recall the following result, see \cite{A, GT, MT}.

\begin{thm}\label{thm:embedding}
 A rotation $\rho$ on $\G$ is equivalent to a cellular embedding of $\G$ in an orientable surface.
 The face boundaries of the embedding corresponding to $\rho$ are the orbits of $\rho \circ \tau$.
\end{thm}

Given a  relative Heffter array $A=\H_t(m,n;s,k)$, the orderings $\omega_r$ and $\omega_c$ are said to be
\emph{compatible} if
$\omega_c \circ \omega_r$ is a cycle of length $|\E(A)|$.

\begin{thm}\label{thm:biembedding}
Let $A$ be a  relative Heffter array $\H_t(m,n;s,k)$ that is simple with respect to the compatible orderings $\omega_r$
and $\omega_c$.
Then there exists a cellular biembedding of the cyclic cycle decompositions $\mathcal{D}_{\omega_r^{-1}}$ and
$\mathcal{D}_{\omega_c}$ of $K_{\frac{2nk+t}{t}\times t}$ into an orientable surface of genus
$$g=1+\frac{(nk-n-m-1)(2nk+t)}{2}.$$
\end{thm}

\begin{proof}
Since the orderings $\omega_r$ and $\omega_c$ are compatible, we have that $\omega_c\circ\omega_r$ is a cycle of length
$|\E(A)|$. Let us consider the permutation $\bar{\rho}_0$ on $\pm \E(A)=\mathbb{Z}_{2nk+t}\setminus
\frac{2nk+t}{t}\mathbb{Z}_{2nk+t}$, where $\frac{2nk+t}{t}\mathbb{Z}_{2nk+t}$ denotes the subgroup of $\mathbb{Z}_{2nk+t}$ of order $t$, defined by:
$$\bar{\rho}_0(a)=\begin{cases}
-\omega_r(a)\mbox{ if } a\in \E(A);\\
\omega_c(-a)\mbox{ if } a\in -\E(A).\\
\end{cases}$$
Note that, if $a\in \E(A)$, then $\bar{\rho}_0^2(a)=\omega_c\circ\omega_r(a)$ and hence $\bar{\rho}_0^2$ acts
cyclically on $\E(A)$. Also $\bar{\rho}_0$ exchanges $\E(A)$ with $-\E(A)$. Thus it acts cyclically on $\pm \E(A)$.

We note that the graph $K_{\frac{2nk+t}{t}\times t}$ is nothing but $\Cay[\mathbb{Z}_{2nk+t}:\mathbb{Z}_{2nk+t}\setminus
\frac{2nk+t}{t}\mathbb{Z}_{2nk+t}]$ that is $\Cay[\mathbb{Z}_{2nk+t}:\pm \E(A)]$.
Now, we define the map $\rho$ on the set of the oriented edges of the Cayley graph $\Cay[\mathbb{Z}_{2nk+t}:\pm \E(A)]$
so that:
$$\rho((x,x+a))= (x,x+\bar{\rho}_0(a)).$$
Since $\bar{\rho}_0$ acts cyclically on $\pm \E(A)$ the map $\rho$ is a rotation of $\Cay[\mathbb{Z}_{2nk+t}:\pm
\E(A)]$.
Hence, by Theorem \ref{thm:embedding}, there exists a cellular embedding $\sigma$ of  $\Cay[\mathbb{Z}_{2nk+t}:\pm \E(A)]$ in an
orientable surface
so that the face boundaries correspond to the orbits of $\rho\circ \tau$ where $\tau((x,x+a))=(x+a,x)$.
Let us consider the oriented edge $(x,x+a)$ with $a \in \E(A)$, and let  $\overline{C}$ be the column containing $a$.
Since $a\in \E(A)$, $-a\in -\E(A)$ and we have that:
$$\rho\circ\tau((x,x+a))=\rho((x+a,(x+a)-a))=(x+a,x+a+\omega_c(a)).$$
Thus $(x,x+a)$ belongs to the boundary of the face $F_1$ delimited by the oriented edges:
$$(x,x+a),(x+a,x+a+\omega_c(a)),(x+a+\omega_c(a),x+a+\omega_c(a)+\omega_c^2(a)),\dots$$
$$\dots ,\left(x+\sum_{i=0}^{|\E(\overline{C})|-2} \omega_c^i(a),x\right).$$
We note that the cycle associated to the face $F_1$ is:
$$\left(x,x+a,x+a+\omega_c(a),\ldots,x+\sum_{i=0}^{|\E(\overline{C})|-2} \omega_c^i(a)\right).$$
Let us now consider the oriented edge $(x,x+a)$ with $a\not \in \E(A)$.
Hence $-a\in \E(A)$, and we name by $\overline{R}$ the row containing the element $-a$. Since $-a\in \E(A)$ we have
that:
$$\rho\circ\tau((x,x+a))=\rho((x+a,(x+a)-a))=(x+a,x+a-\omega_r(-a)).$$
Thus $(x,x+a)$ belongs to the boundary of the face $F_2$ delimited by the oriented edges:
$$(x,x+a),(x-(-a),x-(-a)-\omega_r(-a)),$$
$$(x-(-a)-\omega_r(-a),x-(-a)-\omega_r(-a)-\omega_r^2(-a)),\dots, \left(x-\sum_{i=0}^{|\E(\overline{R})|-2}
\omega_r^i(-a),x\right).$$
Since $A$ is a Heffter array and $\omega_r$ acts cyclically on $\E(\overline{R})$, for any $j\in [1,
|\E(\overline{R})|]$ we have that:
$$-\sum_{i=0}^{j-1}\omega_{r}^{i}(-a)=
\sum_{i=j}^{|\E(\overline{R})|-1}\omega_{r}^{i}(-a)=\sum_{i=1}^{|\E(\overline{R})|-j}\omega_{r}^{|\E(\overline{R})|-i}
(-a)=\sum_{i=1}^{|\E(\overline{R})|-j}\omega_{r}^{-i}(-a).$$
It follows that the cycle associated to the face $F_2$ can be written also as:
$$\left(x,x+\sum_{i=1}^{|\E(\overline{R})|-1}\omega_{r}^{-i}(-a),x+\sum_{i=1}^{|\E(\overline{R})|-2}\omega_{r}^{-i}(-a),
\dots,x+\omega_{r}^{-1}(-a)\right).$$
Therefore any nonoriented edge $\{x,x+a\}$ belongs to the  boundaries of exactly two faces: one of type $F_1$ and one of
type $F_2$. Hence the embedding is 2-colorable.

Moreover, it is easy to see that those face boundaries are the cycles obtained from the relative Heffter array $A$ following the
orderings  $\omega_c$ and $\omega_r^{-1}$.

To calculate the genus $g$ it is sufficient to recall that $V-S+F=2-2g$, where $
V$, $S$ and $F$ denote the number of vertices,
edges and faces determined by the embedding on the surface, respectively.
We have $V=2nk+t$, $S=nk(2nk+t)$ and $F=(2nk+t)(n+m)$.
\end{proof}

Looking for compatible orderings in the case of a globally simple Heffter array led us to investigate the following problem
introduced in \cite{CDP}.
Let $A$ be an $m\times n$ \emph{toroidal} p.f. array. By $r_i$ we denote the orientation of the $i$-th row,
precisely $r_i=1$ if it is from left to right and $r_i=-1$ if it is from right to left. Analogously, for the $j$-th
column, if its orientation $c_j$ is  from  top to bottom then $c_j=1$ otherwise $c_j=-1$. Assume that an orientation
$\R=(r_1,\dots,r_m)$
 and $\C=(c_1,\dots,c_n)$ is fixed. Given an initial filled cell $(i_1,j_1)$ consider the sequence
$ L_{\R,\C}(i_1,j_1)=((i_1,j_1),(i_2,j_2),\ldots,(i_\ell,j_\ell),$ $(i_{\ell+1},j_{\ell+1}),\ldots)$
where $j_{\ell+1}$ is the column index of the filled cell $(i_\ell,j_{\ell+1})$ of the row $\overline{R}_{i_\ell}$ next
to
$(i_\ell,j_\ell)$ in the orientation $r_{i_\ell}$,
and where $i_{\ell+1}$ is the row index of the filled cell of the column $\overline{C}_{j_{\ell+1}}$ next to
$(i_\ell,j_{\ell+1})$ in the orientation $c_{j_{\ell+1}}$.
The problem is the following:

\begin{KN}
Given a toroidal p.f. array $A$,
do there exist $\R$ and $\C$ such that the list $L_{\R,\C}$ covers all the filled
cells of $A$?
\end{KN}

By $P(A)$ we will denote the \probname\ for a given array $A$.
Also, given a filled cell $(i,j)$, if $L_{\R,\C}(i,j)$ covers all the filled positions of $A$ we will
say that  $(\R,\C)$ is a solution of $P(A)$.
For known results about this problem see \cite{CDP}.
The relationship between the Crazy Knight's Tour Problem and globally simple relative Heffter arrays is explained in the following result
which is an easy consequence of Theorem
\ref{thm:biembedding}.

\begin{cor}\label{preprecedente}
  Let $A$ be a globally simple relative Heffter array $\H_t(m,n;s,k)$  such that $P(A)$ admits a solution $(\R,\C)$.
  Then there exists a biembedding of the cyclic cycle decompositions $\mathcal{D}_{\omega_r^{-1}}$ and
$\mathcal{D}_{\omega_c}$ of $K_{\frac{2nk+t}{t}\times t}$ into an orientable surface.
\end{cor}

Extending \cite[Theorem 1.1]{CDDYbiem} to the relative case, we have the following result (see also \cite[Theorem 2.7]{CDP}).

\begin{prop}
If there exist compatible simple orderings $\omega_r$ and $\omega_c$ for a $\H_t(m,n;s,k)$, then one of the following
cases occurs:
\begin{itemize}
 \item[(1)] $m,n,s,k$ are all odd;
 \item[(2)] $m$ is odd and $n,k$ are even;
 \item[(3)] $n$ is odd and $m,t$ are even.
\end{itemize}
\end{prop}

Given a positive integer $n$,
let $0<\ell_1<\ell_2<\ldots<\ell_k<n$ be integers.
We denote by $A_n=A_n(\ell_1,\ell_2,\ldots,\ell_k)$ a $k$-diagonal p.f. array of size $n$ whose filled diagonals are
$D_{\ell_1},D_{\ell_2},\ldots,D_{\ell_k}$.
Let $M=\lcm(\ell_2-\ell_1,\ell_3-\ell_2,\dots,\ell_{k}-\ell_{k-1},\ell_k-\ell_1)$ and set
$A_{n+M}=A_{n+M}(\ell_1,\ell_2,\ldots,\ell_k)$.
We now study the Crazy Knight's Tour Problem for such arrays $A_n$. As a consequence, we will obtain new
biembeddings of cycle decompositions of complete graphs on orientable surfaces.

\begin{thm}\label{Simone}
Suppose that the problem $P(A_n)$ admits a solution $(\R,\C)$ where
$\R=(1,1,$ $\dots,1)$ and $\C=(c_1,c_2,\dots,c_{n-\ell_k+1},1,1,\ldots,1)$.
Then  $P(A_{n+M})$ admits the solution $(\R',\C')$ where
$\R'=(1,1,\ldots,1)$ and $\C'=(c_1,c_2,\dots,c_{n-\ell_k+1},1,1,\ldots,1)$.
\end{thm}

\begin{proof}
We denote by $E$ the set of indices $i$ such that $c_i=-1$
and by $B_n$ the p.f. array of size $n$ obtained from $A_n$ by replacing each column $\overline{C}_j$, when $j\not \in E$, with
an empty column.
Also, we denote by $B_{n+M}$ the p.f. array of size $n+M$ obtained from $A_{n+M}$ in the same way using the same set $E$.
As $E\subseteq [1,n-\ell_{k}+1]$, the nonempty cells of $B_n$ are of the form $((e-1)+\ell_i,e)$ for $e\in E$ and $i\in [1,k]$.
Since $(e-1)+\ell_i\leq n$, we have $skel(B_n)=skel(B_{n+M})$.
So we can set $B=skel(B_{n})=skel(B_{n+M})$.

For any $x=(i_1,j_1) \in B$, consider the sequence $X=L_{\R,\C}(i_1,j_1)$ defined on $skel(A_n)$ and let $y$ be the second element of $X$ that belongs to $B$
if $|X \cap B|\geq 2$, $y=x$ otherwise.
Define  $\vartheta_n: B\to B$ by setting $\vartheta_n(x)=y$.
Take $(\R',\C')$ as in the statement and
define the map $\vartheta_{n+M}: B \to B$ as before considering the sequence $L_{\R',\C'}(x)$ defined on $skel(A_{n+M})$.

In order to prove that $\vartheta_n(x)=\vartheta_{n+M}(x)$, for any $h\in [1,k]$, we set:
$$ \sigma(h)=\left\{\begin{array}{ll}
 \ell_1-\ell_{k-1} & \mbox{ if } h=1;\\
 \ell_2-\ell_{k} & \mbox{ if } h=2;\\
\ell_h-\ell_{h-2} &  \mbox{ otherwise}
\end{array}\right.\quad \textrm{ and }
\quad  \delta(h)=\left\{\begin{array}{ll}
\ell_1-\ell_{k}  & \mbox{ if } h=1;\\
\ell_h-\ell_{h-1} & \mbox{ otherwise}.
\end{array}\right.$$

Set $x=(i_1,j_1)\in B$, hence $x\in D_{\ell_h}$ for some $h\in [1,k]$.
We have that $\vartheta_n(x)=(i_1+\delta(h)\lambda-\sigma(h),j_1+\delta(h)\lambda)\pmod{n}$
where $\lambda$ is the minimum positive integer such that $(j_1+\delta(h)\lambda)\pmod{n}\in E$.
Similarly $\vartheta_{n+M}(x)=(i_1+\delta(h)\lambda'-\sigma(h),j_1+\delta(h)\lambda')\pmod{n+M}$
where $\lambda'$ is the minimum positive integer such that $(j_1+\delta(h)\lambda')\pmod{n+M}\in E$.
Write  $j_1+\delta(h)\lambda=qn+r$ where $1\leq r\leq n$, which means $r\in E$.

If $q=0$, we clearly have  $\lambda'=\lambda$ and hence $\vartheta_{n+M}(x)=\vartheta_n(x)$.
Otherwise, since the last $M$ elements of $\C'$ are equal to $1$, we have that $\lambda'=\lambda+\frac{qM}{\delta(h)}$. Hence:
$$\begin{array}{rcl}
\vartheta_{n+M}(x) & =& \left(i_1+\delta(h)\left(\lambda+\frac{qM}{\delta(h)}\right)-\sigma(h),
j_1+\delta(h)\left(\lambda+\frac{qM}{\delta(h)}\right)\right)\pmod{n+M}\\
 & =& (i_1+\delta(h)\lambda+qM-\sigma(h),j_1+\delta(h)\lambda+qM)\pmod{n+M}\\
 & =& ((i_1-j_1)+q(n+M)+r-\sigma(h),q(n+M)+r)\pmod{n+M}\\
 & =& ((i_1-j_1)+r-\sigma(h),r)\pmod{n+M}.
\end{array}$$

It is not hard to see that $1\leq (i_1-j_1)+r-\sigma(h)\leq n$; also recall that $1\leq r\leq n$.
Hence
$$\vartheta_{n+M}(x)=((i_1-j_1)+r-\sigma(h),r).$$
On the other hand, by  $j_1+\delta(h)\lambda=qn+r$, we obtain:
$$((i_1-j_1)+r-\sigma(h),r)=(i_1+\delta(h)\lambda-\sigma(h),j_1+\delta(h)\lambda)\pmod{n}=\vartheta_n(x).$$
So we have proved that $\vartheta_{n+M}(x)=\vartheta_{n}(x)$ for any $x\in B$.

For any $(i,j)\in skel(A_n)$, since $(\R,\C)$ is a solution of $P(A_n)$, we have
 $L_{\R,\C}(i,j)\cap B=B$.
 Moreover, since $\vartheta_n(x)=\vartheta_{n+M}(x)$ for any $x\in B$, it follows that for any $(i', j')\in skel(A_{n+M})$ we have
  $L_{\R',\C'}(i',j')\cap B$ is either $B$ or $\emptyset$.
  If there exists $(\bar{\imath},\bar{\jmath})\in skel(A_{n+M})$ such that
  $L_{\R',\C'}(\bar\imath,\bar\jmath)\cap B=\emptyset$ then for any $\lambda' \in \mathbb{N}$,
  the cell $(\bar\imath+\delta(\bar{h})\lambda',\bar\jmath+\delta(\bar{h})\lambda')\pmod{n+M}$ is not in $B$.
  On the other hand
there exists $\lambda \in \mathbb{N}$, such that
 $(\bar\imath+\delta(\bar{h})\lambda,\bar\jmath+\delta(\bar{h})\lambda)\pmod{n}\in B$,
since $(\R,\C)$ is a solution of $P(A_n)$.
Also, since $\delta(\bar{h})$ divides $M$ there exists $\bar{q}\in \mathbb{N}$ such that
  $(\bar\imath+\delta(\bar{h})\bar\lambda,\bar\jmath+\delta(\bar{h})\bar\lambda)\pmod{n+M}\in B$, where
	$\bar \lambda=\lambda + \bar q M / \delta(\bar h)$.
  Hence
   $L_{\R',\C'}(\bar{\imath},\bar{\jmath})\cap B\neq\emptyset$, which is a contradiction.
Thus it follows that $(\R',\C')$ is a solution of $P(A_{n+M})$.
\end{proof}

\begin{cor}\label{precedente}
Let $k\equiv 3 \pmod 4$ and $n\equiv 1\pmod {4}$ be such that $n\geq k$ and $3\leq k\leq 119$.
Let $A_n$ be a $k$-diagonal array whose filled diagonals are $D_1,D_2,\ldots,D_{k-3}$, $D_{k-1}$,
$D_{k}$ and $D_{k+1}$. Then $P(A_n)$ admits a solution.
\end{cor}

\begin{proof}
Let $k=4h+3$ and $M=\lcm(2,4h+3)$, that is $M=2(4h+3)$.
For any $1\leq h\leq 29$,  with the help of a computer, we have checked the existence of a solution of $P(A_n)$ for any $n\in[4h+5,4h+5+M]=[4h+5,12h+11]$,
that satisfies the hypothesis of Theorem \ref{Simone}.
Hence the claim follows by this theorem.
\end{proof}

\begin{cor}
  Let $k\equiv 3 \pmod 4$ and $n\equiv 1\pmod {4}$ such that  $n\geq k$ and $3\leq k\leq 119$.
  Then there exists a globally simple $\H_1(n;k)$ with
  orderings $\omega_r$ and $\omega_c$ which are both simple and compatible.
  As a consequence, there exists a biembedding of cyclic $k$-cycle decompositions of the complete graph
  $K_{2nk+1}$   into an orientable surface.
  \end{cor}

\begin{proof}
The existence of a globally simple $\H_1(n;k)$, whose filled diagonals are $D_1,D_2,\ldots,$
$D_{k-3},D_{k-1},D_{k}, D_{k+1}$,
 was proven in \cite{BCDY}.
The result follows from Corollaries \ref{preprecedente} and \ref{precedente}.
\end{proof}

\section{Direct constructions of globally simple $\H_t(n;k)$}\label{sec:constructions}

Many of the constructions we will present  are based on filling in the cells of a set of diagonals.
In order to describe these constructions we use the same procedure introduced in \cite{DW}.
In an $n\times n$ array $A$ the procedure $diag(r,c,s,\Delta_1,\Delta_2,\ell)$ installs the entries
$$A[r+i\Delta_1,c+i\Delta_1]=s+i\Delta_2\ \textrm{for}\ i\in[0,\ell-1],$$
where by $A[i,j]$ we mean the element of $A$ in position $(i,j)$.
The parameters used in the $diag$ procedure have the following meaning:
\begin{itemize}
  \item $r$ denotes the starting row,
  \item $c$ denotes the starting column,
  \item $s$ denotes the entry $A[r,c]$,
  \item $\Delta_1$ denotes the increasing value of the row and column at each step,
  \item $\Delta_2$ denotes how much the entry is changed at each step,
  \item $\ell$ is the length of the chain.
\end{itemize}
We will write $[a,b]_{(W)}$ to mean $supp(W)=[a,b]$.

\begin{prop}\label{prop:3-t=n}
For every odd $n\geq 3$ there exists an integer cyclically $3$-diagonal $\H_{n}(n;3)$.
\end{prop}

\begin{proof}
We construct an $n\times n$ array $A$ using the following procedures labeled $\texttt{A}$ to $\texttt{E}$:
$$\begin{array}{ll}
\texttt{A}:\;   diag\left(1,1, -\frac{7n-9}{2} ,  1 ,7, n\right); &
\texttt{B}:\;   diag\left(1,2, \frac{7n-3}{2}  ,   2,  -7 , \frac{n+1}{2}\right);\\[3pt]
\texttt{C}:\;   diag\left(2,3, -5, 2, -7 , \frac{n-1}{2}\right); &
\texttt{D}:\;   diag\left(2,1, \frac{7n-13}{2}, 2, -7 ,\frac{n+1}{2}\right);\\[3pt]
\texttt{E}:\;   diag\left(3,2, -10, 2, -7 ,\frac{n-1}{2}\right). &
\end{array}$$

We prove that the array constructed above is an integer  cyclically $3$-diagonal $\H_{n}(n;3)$.
To aid in the proof we give a schematic picture of where each of the diagonal procedures fills cells (see Figure \ref{fig2}).
Note that each row and each column contain exactly $3$ elements. We now check that the elements in every row sum to zero
(in $\Z$).

\begin{figure}[ht]
\begin{footnotesize}
 \begin{center}
\begin{tabular}{|c|c|c|c|c|c|c|c|c|} \hline
 \texttt{A} & \texttt{B} &  &  &  & & &   & \texttt{D}\\ \hline
 \texttt{D} & \texttt{A} & \texttt{C} & & & & & &  \\ \hline
 &\texttt{E} & \texttt{A} & \texttt{B} & & & & &     \\ \hline
 & &\texttt{D} & \texttt{A} & \texttt{C}  & & & &    \\ \hline
 & & &\texttt{E} & \texttt{A} & \texttt{B} & & &     \\ \hline
 & & & &\texttt{D} & \texttt{A} & \texttt{C} & &     \\ \hline
 & & & & &\texttt{E} & \texttt{A} & \texttt{B} &      \\ \hline
 & & & & & &\texttt{D} & \texttt{A} & \texttt{C}     \\ \hline
\texttt{B} & & & & & & &\texttt{E} & \texttt{A}     \\ \hline
\end{tabular}
\end{center}
 \end{footnotesize}
\caption{Scheme of construction with $n=9$.}\label{fig2}
\end{figure}

\begin{description}
\item[Row $1$] There is the first value of the \texttt{A} diagonal and of the \texttt{B} diagonal and the last of the
\texttt{D} diagonal.
The sum is
$$-\frac{7n-9}{2}+\frac{7n-3}{2}-3=0.$$
\item[Row $2$  to $n$] There are two cases depending on whether the row $r$ is even or odd.
If $r$ is even, then write $r=2i+2$ where $i\in\left[0,\frac{n-3}{2}\right]$. Notice that from the \texttt{D},
\texttt{A} and \texttt{C} diagonal cells we get the following sum:
$$\left(\frac{7n-13}{2}-7i\right)+\left(-\frac{7n-23}{2}+14i\right)+\left(-5-7i\right)=0.$$
If $r$ is odd, then write $r=2i+3$ where $i\in\left[0,\frac{n-3}{2}\right]$.
From the \texttt{E}, \texttt{A} and \texttt{B} diagonal cells we get the following sum:
$$\left(-10-7i\right)+\left(-\frac{7n-37}{2}+14i\right)+\left(\frac{7n-17}{2}-7i\right)=0.$$
\end{description}
So we have shown that all row sums are zero. Next we check that the columns all add to zero.
\begin{description}
\item[Column $1$] There is the first value of the \texttt{A} diagonal and of the \texttt{D} diagonal and the last of the
\texttt{B} diagonal.
The sum is $$-\frac{7n-9}{2}+\frac{7n-13}{2}+2=0.$$
\item[Column $2$  to $n$] There are two cases depending on whether the column $c$ is even or odd.
If $c$ is even, then write $c=2i+2$ where $i\in\left[0,\frac{n-3}{2}\right]$.
Notice that from the \texttt{B}, \texttt{A} and \texttt{E} diagonal cells we get the following sum:
$$\left(\frac{7n-3}{2}-7i\right)+\left(-\frac{7n-23}{2}+14i\right)+\left(-10-7i\right)=0.$$
If $c$ is odd, then write $c=2i+3$ where $i\in\left[0,\frac{n-3}{2}\right]$.
From the \texttt{C}, \texttt{A} and \texttt{D} diagonal cells we get the following sum:
$$\left(-5-7i\right)+\left(-\frac{7n-37}{2}+14i\right)+\left(\frac{7n-27}{2}-7i\right)=0.$$
\end{description}
So we have shown that each column sums to zero.
Also, it is not hard to see that:
$$supp(\texttt{A})=\left\{1,8,15,\ldots,\frac{7n-5}{2}\right\}\cup \left\{6,13,20,\ldots,\frac{7n-9}{2} \right\}, $$
\begin{align*}
& supp(\texttt{B})=\left\{2,9,16,\ldots,\frac{7n-3}{2}\right\}, &
supp(\texttt{C})=\left\{5,12,19,\ldots,\frac{7n-11}{2}\right\},\\
& supp(\texttt{D})=\left\{3 \right\}\cup \left\{4,11,18,\ldots,\frac{7n-13}{2}\right\}, &
supp(\texttt{E})=\left\{10,17,24,\ldots,\frac{7n-1}{2}\right\},
\end{align*}
hence $supp(A)=\left[1,\frac{7n-1}{2}\right]\setminus\left\{7,14,21,\ldots,\frac{7n-7}{2}\right\}$.
This concludes the proof.
\end{proof}

\begin{ex}\label{example42}
Following the proof of Proposition \ref{prop:3-t=n} we obtain the
integer $\H_9(9;3)$ below.
 \begin{center}
\begin{footnotesize}
$\begin{array}{|r|r|r|r|r|r|r|r|r|}
\hline -27 & 30 &   &  &  &  &  & & -3 \\
\hline  25 & -20 & -5 &  &  &  &  &  &  \\
\hline   & -10 & -13 & 23 &   &  &  & & \\
\hline   &   & 18 & -6 & -12 &  &  &  & \\
\hline   &  &  &   -17 & 1 & 16  & &  &\\
\hline   &  &  &   &   11 & 8 & -19 & &  \\
\hline    &  &  &  &  & -24 & 15 & 9 &   \\
\hline   &  &  &  &  &   & 4 & 22  & -26 \\
\hline    2 &  &   &  &  &  & & -31 & 29    \\
\hline
\end{array}$
\end{footnotesize}
\end{center}

We can use this example to briefly explain how the construction has been obtained (a similar idea will be used also 
in Proposition \ref{prop:3-t=2n} below). First of all, we have to avoid the multiples of $\frac{2nk}{t}+1=7$, so
we work modulo $7$.
The diagonal $D_1$ consists of elements, all congruent to $1$ modulo $7$, arranged in arithmetic progression where,
for instance, the central cell is filled with $1$.
The other two filled diagonals are obtained in such a way that the elements of 
$D_9$ are all congruent to $2$ modulo $7$ and the elements of $D_2$ are all congruent to $-3$ modulo $7$.
This can be achieved filling the cell $(9,1)$ with the integer $2$: it is now easy to obtain the elements in the remaining cells, 
remembering that the row/column sums are $0$.
\end{ex}

\begin{prop}\label{prop:3-t=2n}
For every odd $n\geq 3$ there exists an integer cyclically $3$-diagonal $\H_{2n}(n;3)$.
\end{prop}

\begin{proof}
We construct an $n\times n$ array $A$ using the following procedures labeled $\texttt{A}$ to $\texttt{E}$:
$$\begin{array}{ll}
\texttt{A}:\;   diag\left(1,1, -(4n-5) ,  1 ,8, n\right); &
\texttt{B}:\;   diag\left(1,2, 4n-2  ,   2,  -8 , \frac{n+1}{2}\right);\\[3pt]
\texttt{C}:\;   diag\left(2,3, -6, 2, -8 , \frac{n-1}{2}\right); &
\texttt{D}:\;   diag\left(2,1, 4n-7, 2, -8 ,\frac{n+1}{2}\right);\\[3pt]
\texttt{E}:\;   diag\left(3,2, -11, 2, -8 ,\frac{n-1}{2}\right). &
\end{array}$$

We  prove that the array constructed above is an integer  cyclically $3$-diagonal $\H_{2n}(n;3)$.
To aid in the proof we give a schematic picture of where each of the diagonal procedures fills cells (see Figure \ref{fig2}).
Note that each row and each column contain exactly $3$ elements. We now check that the elements in every row sum to zero
(in $\Z$).
\begin{description}
\item[Row $1$] There is the first value of the \texttt{A} diagonal and of the \texttt{B} diagonal and the last of the
\texttt{D} diagonal.
The sum is
$$-(4n-5)+(4n-2)-3=0.$$
\item[Row $2$  to $n$] There are two cases depending on whether the row $r$ is even or odd.
If $r$ is even, then write $r=2i+2$ where $i\in\left[0,\frac{n-3}{2}\right]$. Notice that from the \texttt{D},
\texttt{A} and \texttt{C} diagonal cells we get the following sum:
$$\left(4n-7-8i\right)+\left(-4n+13+16i\right)+(-6-8i)=0.$$
If $r$ is odd, then write $r=2i+3$ where $i\in\left[0,\frac{n-3}{2}\right]$.
From the \texttt{E}, \texttt{A} and \texttt{B} diagonal cells we get the following sum:
$$\left(-11-8i\right)+\left(-4n+21+16i\right)+\left(4n-10-8i\right)=0.$$
\end{description}
So we have shown that all row sums are zero. Next we check that the columns all add to zero.
\begin{description}
\item[Column $1$] There is the first value of the \texttt{A} diagonal and of the \texttt{D} diagonal and the last of the
\texttt{B} diagonal.
The sum is $$-(4n-5)+(4n-7)+2=0.$$
\item[Column $2$  to $n$] There are two cases depending on whether the column $c$ is even or odd.
If $c$ is even, then write $c=2i+2$ where $i\in\left[0,\frac{n-3}{2}\right]$.
Notice that from the \texttt{B}, \texttt{A} and \texttt{E} diagonal cells we get the following sum:
$$\left(4n-2-8i\right)+\left(-4n+13+16i\right)+\left(-11-8i\right)=0.$$
If $c$ is odd, then write $c=2i+3$ where $i\in\left[0,\frac{n-3}{2}\right]$.
From the \texttt{C}, \texttt{A} and \texttt{D} diagonal cells we get the following sum:
$$\left(-6-8i\right)+\left(-4n+21+16i\right)+\left(4n-15-8i\right)=0.$$
\end{description}
So we have shown that each column sums to zero.
Also, it is not hard to see that:
$$supp(\texttt{A})=\left\{1,9,17,\ldots,4n-3\right\}\cup \left\{7,15,23,\ldots,4n-5 \right\}, $$
\begin{align*}
& supp(\texttt{B})=\left\{2,10,18,\ldots,4n-2\right\}, &
supp(\texttt{C})=\left\{6,14,22,\ldots,4n-6\right\},\\
& supp(\texttt{D})=\left\{3 \right\}\cup \left\{5,13,21,\ldots,4n-7\right\}, &
supp(\texttt{E})=\left\{11,19,27,\ldots,4n-1\right\},
\end{align*}
hence $supp(A)=\left[1,4n-1\right]\setminus\left\{4,8,12,\ldots,4n-4\right\}$.
This concludes the proof.
\end{proof}

\begin{ex}
Following the proof of Proposition \ref{prop:3-t=2n} we obtain the
integer $\H_{18}(9;3)$ below.
 \begin{center}
\begin{footnotesize}
$\begin{array}{|r|r|r|r|r|r|r|r|r|}
\hline -31 & 34 &   &  &  &  &  & & -3 \\
\hline  29 & -23 & -6 &  &  &  &  &  &  \\
\hline   & -11 & -15 & 26 &   &  &  & & \\
\hline   &   & 21 & -7 & -14 &  &  &  & \\
\hline   &  &  &   -19 & 1 & 18  & &  &\\
\hline   &  &  &   &   13 & 9 & -22 & &  \\
\hline    &  &  &  &  & -27 & 17 & 10 &   \\
\hline   &  &  &  &  &   & 5 & 25  & -30 \\
\hline   2  &  &   &  &  &  & & -35 & 33    \\
\hline
\end{array}$
\end{footnotesize}
\end{center}
\end{ex}

In the following propositions, since $k>5$, in order to prove that the
relative Heffter array $\H_k(n;k)$ constructed is globally simple we have to
show that the partial sums of each row and of each column are distinct modulo $2nk+k$.
From now on, the sets $\E(\overline{R}_i)$ and $\E(\overline{C}_i)$ are considered ordered with respect to the natural ordering.
Also, by $\S(\overline{R}_i)$ and $\S(\overline{C}_i)$ we will denote the sequence of the partial sums
 of  $\E(\overline{R}_i)$ and $\E(\overline{C}_i)$, respectively.
In order to check that the partial sums are distinct the following remark allows to reduce the computations.

 \begin{rem}\label{rem:partialsum}
Let $A$ be a $\H_t(n;k)$. By the definition of a (relative) Heffter array it easily follows that the $i$-th partial sum $s_i$ of a row (or a column) is
different from the partial sums $s_{i-2}, s_{i-1}, s_{i+1}$ and $s_{i+2}$ of the same row (column).
 \end{rem}

\begin{prop}\label{prop:7}
For every $n\geq 7$ with $n\equiv 3 \pmod 4$ there exists an integer cyclically $7$-diagonal globally simple $\H_{7}(n;7)$.
\end{prop}

\begin{proof}
We construct an $n\times n$ array $A$ using the following procedures labeled $\texttt{A}$ to $\texttt{N}$:
$$\begin{array}{ll}
\texttt{A}:\;   diag\left(3,3,-\frac{n+1}{2},2,-1,\frac{n-1}{2}\right); &
\texttt{B}:\;   diag\left(4,4,1,2,1,\frac{n-3}{2}\right);\\[3pt]
\texttt{C}:\;   diag\left(n-2,n-1,-(5n+3),2,-1,n\right); &
\texttt{D}:\;   diag\left(2,1,-(4n+3),2,-1,n\right);\\[3pt]
\texttt{E}:\;   diag\left(1,3,\frac{7n+3}{4},4,1,\frac{n+1}{4}\right); &
\texttt{F}:\;   diag\left(2,4,\frac{3n+1}{2},4, -1, \frac{n+1}{4}\right);\\[3pt]
\texttt{G}:\;   diag\left(3,5,\frac{11n+7}{4},4,1,\frac{n+1}{4}\right); &
\texttt{H}:\;   diag\left(4,6,\frac{5n+1}{2},4,-1,\frac{n-3}{4}\right);\\[3pt]
\texttt{I}:\;   diag\left(3,1,-\frac{9n+5}{4},4,1,\frac{n+1}{4}\right); &
\texttt{J}:\;   diag\left(4,2,-\frac{5n+3}{2},4,-1,\frac{n+1}{4}\right); \\[3pt]
\texttt{K}:\;   diag\left(5,3,-\frac{5n+1}{4},4,1,\frac{n+1}{4}\right); &
\texttt{L}:\;   diag\left(6,4,-\frac{3n+3}{2},4,-1,\frac{n-3}{4}\right); \\[3pt]
\texttt{M}:\;   diag\left(n-2,1,6n+4,2,1,n\right); &
\texttt{N}:\;   diag\left(2,n-1,3n+2,2,1,n\right). \\[3pt]
\end{array}$$
We also fill the following cells in an \textit{ad hoc} manner:
$$\begin{array}{lcl}
 A\left[1,1\right]=n, & \quad  & A\left[2,2\right]=-\frac{n-1}{2}.
  \end{array}$$

We  prove that the array constructed above is an integer  cyclically $7$-diagonal globally simple $\H_7(n;7)$.
To aid in the proof we give a schematic picture of where each of the diagonal procedures fills cells (see Figure \ref{fig3}).
We have placed an $\texttt{X}$ in the \textit{ad hoc} cells.
Note that each row and each column contains exactly $7$ elements. We now list the elements
and the partial sums of each row.
We leave to the reader the direct check that the partial sums are distinct modulo $14n+7$;
for a quicker check keep in mind Remark \ref{rem:partialsum}.

\begin{figure}
\begin{footnotesize}
 \begin{center}
\begin{tabular}{|c|c|c|c|c|c|c|c|c|c|c|}
  \hline
 \texttt{X} & \texttt{C} & \texttt{E} & \texttt{M} & & & & & \texttt{N} & \texttt{J} & \texttt{D}\\ \hline
 \texttt{D} & \texttt{X} & \texttt{C} & \texttt{F} & \texttt{M} & & & & & \texttt{N} & \texttt{K} \\ \hline
 \texttt{I} &\texttt{D} & \texttt{A} & \texttt{C} & \texttt{G} & \texttt{M} & & & & & \texttt{N}  \\ \hline
 \texttt{N} & \texttt{J} &\texttt{D} & \texttt{B} & \texttt{C}  & \texttt{H} & \texttt{M} & & & &   \\ \hline
 & \texttt{N} & \texttt{K} &\texttt{D} & \texttt{A} & \texttt{C} & \texttt{E} & \texttt{M} & & &    \\ \hline
 & & \texttt{N} & \texttt{L} &\texttt{D} & \texttt{B} & \texttt{C} & \texttt{F} & \texttt{M} & &    \\ \hline
 & & & \texttt{N} & \texttt{I} &\texttt{D} & \texttt{A} & \texttt{C} & \texttt{G} & \texttt{M} &     \\ \hline
 & & & & \texttt{N} & \texttt{J} &\texttt{D} & \texttt{B} & \texttt{C} & \texttt{H} &  \texttt{M}  \\ \hline
 \texttt{M} & & & & & \texttt{N} & \texttt{K} &\texttt{D} & \texttt{A} & \texttt{C} & \texttt{E}   \\ \hline
 \texttt{F}& \texttt{M} & & & & & \texttt{N} & \texttt{L} &\texttt{D} & \texttt{B} & \texttt{C}    \\ \hline
\texttt{C} & \texttt{G} & \texttt{M} & & & & & \texttt{N} & \texttt{I} & \texttt{D} & \texttt{A}     \\ \hline
\end{tabular}
\end{center}
 \end{footnotesize}
\caption{Scheme of construction with $n=11$.}\label{fig3}
\end{figure}

\begin{description}
\item[Row $1$] There is an ad hoc element, the $(\frac{n+5}{2})^{th}$  value of the \texttt{C} diagonal,
the first one of the \texttt{E} diagonal, the $(\frac{n+5}{2})^{th}$  value of the \texttt{M} diagonal,
 the $(\frac{n+1}{2})^{th}$  value of the \texttt{N} diagonal, the last value of the \texttt{J} diagonal
 and the $(\frac{n+1}{2})^{th}$  value of the \texttt{D} diagonal.
Namely,
$$\E(\overline{R}_1)=\left(n,-\frac{11n+9}{2},\frac{7n+3}{4},\frac{13n+11}{2},\frac{7n+3}{2},-\frac{11n+3}{4},-\frac{9n+5}{2}\right)$$
and
$$\S(\overline{R}_1)=\left( n, -\frac{9n+9}{2},-\frac{11n+15}{4}, \frac{15n+7}{4},  \frac{29n+13}{4}, \frac{9n+5}{2},
0 \right).$$

\item[Row $2$] There is the first  value of the \texttt{D} diagonal, an ad hoc element, the third
value of the \texttt{C} diagonal, the first  value of the \texttt{F} diagonal,
 the third  value of the \texttt{M} diagonal, the first value of the \texttt{N} diagonal
 and the last  value of the \texttt{K} diagonal.
Hence
$$\E(\overline{R}_2)=\left(-(4n+3),-\frac{n-1}{2},-(5n+5),\frac{3n+1}{2},6n+6,3n+2,-(n+1)\right) $$
and
$$\S(\overline{R}_2)=\left( -(4n+3), -\frac{9n+5}{2},  -\frac{19n+15}{2}, -(8n+7), -(2n+1), n+1, 0 \right).$$

\item[Row $3$  to $n$] There are four cases depending on the congruence class of $r$ modulo 4.
If $r\equiv 3\pmod 4$, then write $r=4i+3$ where $i\in\left[0,\frac{n-3}{4}\right]$.
It is not hard to see that from the \texttt{N},
\texttt{I}, \texttt{D}, \texttt{A}, \texttt{C}, \texttt{G} and \texttt{M}  diagonal cells we get:
$$\E(\overline{R}_{4i+3})=  \left(\frac{7n+5}{2}+2i, -\frac{9n+5}{4}+i, -\frac{9n+7}{2}-2i, -\frac{n+1}{2}-2i, \right.$$
$$\left. -\frac{11n+11}{2}-2i+\varepsilon,\frac{11n+7}{4}+i, \frac{13n+13}{2}+2i-\varepsilon  \right),$$
where $\varepsilon=0$ for $i\in\left[0,\frac{n-7}{4}\right]$ while $\varepsilon=n$ for $i=\frac{n-3}{4}$,
and
$$\S(\overline{R}_{4i+3})=\left(\frac{7n+5}{2}+2i,\frac{5n+5}{4}+3i,-\frac{13n+9}{4}+i,-\frac{15n+11}{4}-i,\right.$$
$$\left.-\frac{37n+33}{4}-3i+\varepsilon,-\frac{13n+13}{2}-2i+\varepsilon,0 \right). $$
If $r\equiv 0\pmod 4$, then write $r=4i+4$ where $i\in\left[0,\frac{n-7}{4}\right]$.
It is not hard to see that from the \texttt{N},
\texttt{J}, \texttt{D}, \texttt{B}, \texttt{C}, \texttt{H} and \texttt{M}  diagonal cells we get:
$$\E(\overline{R}_{4i+4})=  \left(3n+3+2i,-\frac{5n+3}{2}-i,-(4n+4+2i), 1+2i, -(5n+6+2i), \right.$$
$$\left. \frac{5n+1}{2}-i, 6n+7+2i \right)$$
and
$$\S(\overline{R}_{4i+4})=\left( 3n+3+2i, \frac{n+3}{2}+i,-\frac{7n+5}{2}-i,-\frac{7n+3}{2}+i,-\frac{17n+15}{2}-i,\right.$$
$$\left. -(6n+7+2i),0\right).$$
If $r\equiv 1\pmod 4$, then write $r=4i+5$ where $i\in\left[0,\frac{n-7}{4}\right]$.
It is not hard to see that from the \texttt{N},
\texttt{K}, \texttt{D}, \texttt{A}, \texttt{C}, \texttt{E} and \texttt{M}  diagonal cells we get:
$$\E(\overline{R}_{4i+5})= \left(\frac{7n+7}{2}+2i, - \frac{5n+1}{4}+i, -\frac{9n+9}{2}-2i, -\frac{n+3}{2}-2i, \right.$$
$$\left. -\frac{11n+13}{2}-2i+\varepsilon, \frac{7n+7}{4}+i, \frac{13n+15}{2}+2i-\varepsilon \right),$$
where $\varepsilon=0$ for $i\in\left[0,\frac{n-11}{4}\right]$ while $\varepsilon=n$ for $i=\frac{n-7}{4}$,
and
$$\S(\overline{R}_{4i+5})=\left(\frac{7n+7}{2}+2i,\frac{9n+13}{4}+3i,-\frac{9n+5}{4}+i, -\frac{11n+11}{4}-i,\right.$$
$$\left.  -\frac{33n+37}{4}-3i+\varepsilon, -\frac{13n+15}{2}-2i+\varepsilon,0 \right).$$
If $r\equiv 2\pmod 4$, then write $r=4i+6$ where $i\in\left[0,\frac{n-7}{4}\right]$.
It is not hard to see that from the \texttt{N},
\texttt{L}, \texttt{D}, \texttt{B}, \texttt{C}, \texttt{F} and \texttt{M}  diagonal cells we get:
$$\E(\overline{R}_{4i+6})=\left(3n+4+2i,-\frac{3n+3}{2}-i,-(4n+5+2i),2+2i,-(5n+7+2i),\right. $$
$$\left. \frac{3n-1}{2}-i, 6n+8+2i \right)$$
and
$$\S(\overline{R}_{4i+6})=\left(3n+4+2i, \frac{3n+5}{2}+i,-\frac{5n+5}{2}-i,-\frac{5n+1}{2}+i,-\frac{15n+15}{2}-i,
\right.$$
$$\left. -(6n+8+2i),0 \right).$$
\end{description}
Now we list the elements and the partial sums of the columns.
\begin{description}
\item[Column $1$]  There is an ad hoc element, the first  value of the \texttt{D} diagonal and of the \texttt{I} diagonal,
the second  value of the \texttt{N} diagonal,
 the first  value of the \texttt{M} diagonal, the last value of the \texttt{F} diagonal
 and the second  value of the \texttt{C} diagonal.
Namely,
$$\E(\overline{C}_1)=\left(n,-(4n+3),-\frac{9n+5}{4},3n+3,6n+4,\frac{5n+5}{4},-(5n+4)\right)$$
and
$$\S(\overline{C}_1)=\left(n, -(3n+3),-\frac{21n+17}{4},-\frac{9n+5}{4},\frac{15n+11}{4}, 5n+4, 0\right).$$
\item[Column $2$]  There is the $(\frac{n+5}{2})^{th}$  value of the \texttt{C} diagonal,
an ad hoc element, the  $(\frac{n+3}{2})^{th}$  value of the \texttt{D} diagonal,
the first  value of the \texttt{J} diagonal,
 the $(\frac{n+5}{2})^{th}$  value of the \texttt{N} diagonal and of the \texttt{M} diagonal
 and the last value of the \texttt{G} diagonal.
Namely,
$$\E(\overline{C}_2)=\left(-\frac{11n+9}{2},-\frac{n-1}{2},-\frac{9n+7}{2},-\frac{5n+3}{2},\frac{7n+7}{2},\frac{13n+9}{2},3n+1\right)$$
and
$$\S(\overline{C}_2)=\left(-\frac{11n+9}{2},-(6n+4),-\frac{21n+15}{2},-(13n+9),-\frac{19n+11}{2},-(3n+1), 0\right).$$
\item[Column $3$  to $n$] There are four cases depending on the congruence class of $c$ modulo 4.
If $c\equiv 3\pmod 4$, then write $c=4i+3$ where $i\in\left[0,\frac{n-3}{4}\right]$.
It is not hard to see that from the \texttt{M},
\texttt{E}, \texttt{C}, \texttt{A}, \texttt{D}, \texttt{K} and \texttt{N}  diagonal cells we get:
$$\E(\overline{C}_{4i+3})=\left(6n+5+2i, \frac{7n+3}{4}+i, -(5n+5+2i), -\frac{n+1}{2}-2i,-(4n+4+2i), \right.$$
$$\left.-\frac{5n+1}{4}+i, 3n+4+2i\right)$$
and
$$\S(\overline{C}_{4i+3})=\left(6n+5+2i, \frac{31n+23}{4}+3i,\frac{11n+3}{4}+i,\frac{9n+1}{4}-i, -\frac{7n+15}{4}-3i,\right.$$
$$\left.-(3n+4+2i),0 \right).$$
If $c\equiv 0\pmod 4$, then write $c=4i+4$ where $i\in\left[0,\frac{n-7}{4}\right]$.
It is not hard to see that from the \texttt{M},
\texttt{F}, \texttt{C}, \texttt{B}, \texttt{D}, \texttt{L} and \texttt{N}  diagonal cells we get:
$$\E(\overline{C}_{4i+4})=\left(\frac{13n+11}{2}+2i, \frac{3n+1}{2}-i, -\frac{11n+11}{2}-2i, 1+2i, -\frac{9n+9}{2}-2i, \right.$$
$$\left.  -\frac{3n+3}{2}-i, \frac{7n+9}{2}+2i\right)$$
and
$$\S(\overline{C}_{4i+4})=\left(\frac{13n+11}{2}+2i, 8n+6+i, \frac{5n+1}{2}-i,\frac{5n+3}{2}+i, -(2n+3+i), \right.$$
$$\left. -\frac{7n+9}{2}-2i,0 \right).$$
If $c\equiv 1\pmod 4$, then write $c=4i+5$ where $i\in\left[0,\frac{n-7}{4}\right]$.
It is not hard to see that from the \texttt{M},
\texttt{G}, \texttt{C}, \texttt{A}, \texttt{D}, \texttt{I} and \texttt{N}  diagonal cells we get:
$$\E(\overline{C}_{4i+5})=\left(6n+6+2i, \frac{11n+7}{4}+i, -(5n+6+2i), -\frac{n+3}{2}-2i, -(4n+5+2i),\right.$$
$$\left. -\frac{9n+1}{4}+i, 3n+5+2i\right)$$
and
$$\S(\overline{C}_{4i+5})=\left(6n+6+2i, \frac{35n+31}{4}+3i,\frac{15n+7}{4}+i,\frac{13n+1}{4}-i,-\frac{3n+19}{4}-3i, \right.$$
$$\left. -(3n+5+2i),0\right).$$
If $c\equiv 2\pmod 4$, then write $c=4i+6$ where $i\in\left[0,\frac{n-7}{4}\right]$.
It is not hard to see that from the \texttt{M},
\texttt{H}, \texttt{C}, \texttt{B}, \texttt{D}, \texttt{J} and \texttt{N}  diagonal cells we get:
$$\E(\overline{C}_{4i+6})=\left(\frac{13n+13}{2}+2i, \frac{5n+1}{2}-i, -\frac{11n+13}{2}-2i+\varepsilon, 2+2i, \right.$$
$$\left. -\frac{9n+11}{2}-2i, -\frac{5n+5}{2}-i, \frac{7n+11}{2}+2i-\varepsilon \right),$$
where $\varepsilon=0$ for $i\in\left[0,\frac{n-11}{4}\right]$ while $\varepsilon=n$ for $i=\frac{n-7}{4}$,
and
$$\S(\overline{C}_{4i+6})=\left(\frac{13n+13}{2}+2i, 9n+7+i, \frac{7n+1}{2}-i+\varepsilon,
\frac{7n+5}{2}+i+\varepsilon, -(n+3+i)+\varepsilon,\right.$$
$$\left. -\frac{7n+11}{2}-2i+\varepsilon,0\right).$$
Finally we consider the support of $A$:
$$\begin{array}{rcl}
supp(A)& = & \left[1,\frac{n-3}{2}\right]_{(\texttt{B})}\cup \{ \frac{n-1}{2}\} \cup
\left[\frac{n+1}{2},n-1\right]_{(\texttt{A})} \cup\{ n\} \cup
\left[n+1,\frac{5n+1}{4}\right]_{(\texttt{K})}\cup \\[3pt]
&&\left[\frac{5n+5}{4},\frac{3n+1}{2}\right]_{(\texttt{F})}
 \cup \left[\frac{3n+3}{2},\frac{7n-1}{4}\right]_{(\texttt{L})} \cup
 \left[\frac{7n+3}{4},2n\right]_{(\texttt{E})} \cup
 \left[2n+2,\frac{9n+5}{4}\right]_{(\texttt{I})}\\[3pt]
&& \cup  \left[\frac{9n+9}{4},\frac{5n+1}{2}\right]_{(\texttt{H})}
 \cup \left[\frac{5n+3}{2},\frac{11n+3}{4}\right]_{(\texttt{J})}\cup
 \left[\frac{11n+7}{4},3n+1\right]_{(\texttt{G})}\cup  \\[3pt]
&& \left[3n+2,4n+1\right]_{(\texttt{N})}\cup \left[4n+3,5n+2\right]_{(\texttt{D})}
\cup \left[5n+3,6n+2\right]_{(\texttt{C})} \\[3pt]
&&\cup \left[6n+4,7n+3\right]_{(\texttt{M})} =\\[3pt]
&= & [1,7n+3]\setminus \{2n+1,4n+2,6n+3\}.
\end{array}$$
\end{description}
This concludes the proof.
\end{proof}

\begin{ex}
Following the proof of Proposition \ref{prop:7} we obtain the
integer globally simple $\H_7(11;7)$ below.
 \begin{center}
\begin{footnotesize}
$\begin{array}{|r|r|r|r|r|r|r|r|r|r|r|}
\hline 11 & -65 & 20  & 77 &  &  &  & & 40 & -31 & -52  \\
\hline  -47 & -5 & -60 & 17 & 72 &   &  &   & & 35 & -12\\
\hline  -26 & -53 & -6 & -66 & 32 & 78 &  &  & & & 41\\
\hline  36 & -29 & -48 & 1 & -61  & 28 & 73 & & & & \\
\hline   &  42 & -14 & -54 & -7 & -67 & 21 & 79 & & &\\
\hline   &  &  37 &   -18 & -49 & 2 & -62 & 16 & 74 & &\\
\hline   &  &  &  43 & -25 & -55 & -8 & -68 & 33 & 80 & \\
\hline    &  &  &  & 38 & -30 & -50 & 3 & -63 & 27 & 75\\
\hline  70 &  &  &  &  &  44 & -13 & -56  & -9 & -58 & 22\\
\hline  15 & 76 &  &   &  &  & 39 & -19 & -51 & 4 & -64  \\
\hline  -59  & 34 & 71 &  &  &  &  & 45 & -24 & -57 & -10 \\
\hline
\end{array}$
\end{footnotesize}
\end{center}
\end{ex}

\begin{prop}\label{prop:9}
For every $n\geq 11$ with $n\equiv 3 \pmod 4$ there exists an integer $9$-diagonal globally simple $\H_{9}(n;9)$ with
width $\frac{n-9}{2}$.
\end{prop}

\begin{proof}
We construct an $n\times n$ array $A$ using the following procedures labeled
$\texttt{A}$ to
$\texttt{R}$:
$$\begin{array}{ll}
\texttt{A}:\;  diag(3, 1,   5n+3, 1, 1, n); &
\texttt{B}:\;  diag(4, 1,-(6n+4), 1, -1, n);\\
\texttt{C}:\;  diag(3, 6,-(7n+4), 1, -1, n); &
\texttt{D}:\;  diag(4, 6,   8n+5, 1, 1,n);\\
\texttt{E}:\;  diag\left(1,\frac{n+3}{2}, -(2n), 1, 2,  \frac{n-1}{2}\right); &
\texttt{F}:\;  diag\left(\frac{n+3}{2}, 1, 2n+2, 1, 2,  \frac{n-1}{2}\right);  \\
\texttt{G}:\;  diag\left(2,2,-(n-2), 1, 1 , \frac{n-3}{2}\right); &
\texttt{H}:\;  diag\left(\frac{n+3}{2}, 2, -(2n+3), 1, -2 , \frac{n-3}{2}\right);\\
\texttt{I}:\;  diag\left(2,\frac{n+3}{2}, 2n-1, 1, -2,  \frac{n-3}{2}\right); &
\texttt{J}:\;  diag\left(\frac{n+3}{2},\frac{n+3}{2}, \frac{n-3}{2}, 1, -1,   \frac{n-5}{2}\right); \\
\texttt{K}:\;  diag\left(2,1,-(3n+4), 2, -1,\frac{n+1}{4} \right); &
\texttt{L}:\;  diag\left(1,2, 5n,  2, -1,\frac{n+1}{4} \right); \\
\texttt{M}:\;  diag\left(3,2,-(4n+3), 2, -1 ,\frac{n-3}{4} \right); &
\texttt{N}:\;  diag\left(2,3, 4n+1, 2, -1   ,\frac{n-3}{4} \right); \\
\texttt{O}:\;  diag\left(\frac{n+1}{2},\frac{n+3}{2}, \frac{17n+9}{4}, 2, 1   ,\frac{n-3}{4} \right); &
\texttt{P}:\;  diag\left(\frac{n+3}{2},\frac{n+1}{2},-\frac{15n+7}{4}, 2, 1   ,\frac{n-3}{4} \right); \\
\texttt{Q}:\;  diag\left(\frac{n+3}{2},\frac{n+5}{2},\frac{13n+17}{4}, 2,1   ,\frac{n-3}{4} \right); &
\texttt{R}:\;  diag\left(\frac{n+5}{2},\frac{n+3}{2}, -\frac{19n-1}{4}, 2, 1   ,\frac{n-3}{4} \right).
\end{array}$$
We also fill the following cells in an \textit{ad hoc} manner:
$$\begin{array}{lclcl}
A[1,1]=n-1, & &     A[1,\frac{n+1}{2}]=n+2,  &  &   A[1,n]=-(5n+1),\\
A[\frac{n+1}{2},1]=-(3n), & & A[\frac{n+1}{2},\frac{n+1}{2}]=n, & &  A[\frac{n+1}{2},n]=n+1, \\
A[n-1,n-1]=-\frac{n-1}{2}, & & A[n-1,n]=5n+2, & &  A[n,1]=3n+3,\\
A[n,\frac{n+1}{2}]=-(3n+1), & &   A[n,n-1]=-(3n+2), & & A[n,n]=1.\\
\end{array}$$

We prove that the array constructed above is an integer  $9$-diagonal globally simple $\H_9(n;9)$ with width
$\frac{n-9}{2}$.
To aid in the proof we give a schematic picture of where each of the diagonal procedures fills cells (see Figure \ref{fig9}).
We have placed an $\texttt{X}$ in the \textit{ad hoc} cells.
Note that each row and each column contains exactly $9$ elements.
Since the filled diagonals are $D_1,D_2,D_3,D_4,D_{\frac{n+1}{2}}, D_{\frac{n+3}{2}}, D_{n-2}, D_{n-1}$ and $D_n$, $A$
has two empty strips of size $\frac{n-9}{2}$.
We now list the elements and the partial sums of every row.
We leave to the reader the direct check that the partial sums are distinct modulo $18n+9$;
for a quicker check keep in mind Remark \ref{rem:partialsum}.

\begin{figure}
\begin{footnotesize}
$$\begin{array}{|r|r|r|r|r|r|r|r|r|r|r|r|r|r|r|}\hline
   \texttt{X}&  \texttt{L}&  \texttt{D}&  \texttt{C}&  &  &  &  \texttt{X}&  \texttt{E}&  &   &  &  \texttt{B}&
\texttt{A}&  \texttt{X} \\\hline
    \texttt{K}&  \texttt{G}&  \texttt{N}&  \texttt{D}&  \texttt{C}&  &  &  &  \texttt{I}&  \texttt{E}&          &  &  &
\texttt{B}&  \texttt{A} \\\hline
    \texttt{A}&  \texttt{M}&  \texttt{G}&  \texttt{L}&  \texttt{D}&  \texttt{C}&  &  &  &  \texttt{I}&  \texttt{E}&  &
&  &  \texttt{B} \\\hline
    \texttt{B}&  \texttt{A}&  \texttt{K}&  \texttt{G}&  \texttt{N}&  \texttt{D}&  \texttt{C}&  &  &  &  \texttt{I}&
\texttt{E}&  &  &            \\\hline
    &  \texttt{B}&  \texttt{A}&  \texttt{M}&  \texttt{G}&  \texttt{L}&  \texttt{D}&  \texttt{C}&  &  &     &
\texttt{I}&  \texttt{E}&  &            \\\hline
    &  &  \texttt{B}&  \texttt{A}&  \texttt{K}&  \texttt{G}&  \texttt{N}&  \texttt{D}&  \texttt{C}&  &          &  &
\texttt{I}&  \texttt{E}&            \\\hline
    &  &  &  \texttt{B}&  \texttt{A}&  \texttt{M}&  \texttt{G}&  \texttt{L}&  \texttt{D}&  \texttt{C}&          &  &  &
\texttt{I}&  \texttt{E} \\\hline
    \texttt{X}&  &  &  &  \texttt{B}&  \texttt{A}&  \texttt{K}&  \texttt{X}&  \texttt{O}&  \texttt{D}&  \texttt{C}&  &
&  &  \texttt{X} \\\hline
    \texttt{F}&  \texttt{H}&  &  &  &  \texttt{B}&  \texttt{A}&  \texttt{P}&  \texttt{J}&  \texttt{Q}&  \texttt{D}&
\texttt{C}&  &  &            \\\hline
    &  \texttt{F}&  \texttt{H}&  &  &  &  \texttt{B}&  \texttt{A}&  \texttt{R}&  \texttt{J}&  \texttt{O}&  \texttt{D}&
\texttt{C}&  &            \\\hline
    &  &  \texttt{F}&  \texttt{H}&  &  &  &  \texttt{B}&  \texttt{A}&  \texttt{P}&  \texttt{J}&  \texttt{Q}&
\texttt{D}&  \texttt{C}&            \\\hline
    &  &  &  \texttt{F}&  \texttt{H}&  &  &  &  \texttt{B}&  \texttt{A}&  \texttt{R}&  \texttt{J}&  \texttt{O}&
\texttt{D}&  \texttt{C} \\\hline
    \texttt{C}&  &  &  &  \texttt{F}&  \texttt{H}&  &  &  &  \texttt{B}&  \texttt{A}&  \texttt{P}&  \texttt{J}&
\texttt{Q}&  \texttt{D} \\\hline
    \texttt{D}&  \texttt{C}&  &  &  &  \texttt{F}&  \texttt{H}&  &  &  &  \texttt{B}&  \texttt{A}&  \texttt{R}&
\texttt{X}&  \texttt{X} \\\hline
    \texttt{X}&  \texttt{D}&  \texttt{C}&  &  &  &  \texttt{F}&  \texttt{X}&  &  &      &  \texttt{B}&  \texttt{A}&
\texttt{X}&  \texttt{X} \\\hline
\end{array}$$
\end{footnotesize}

\caption{Scheme of construction with $n=15$.}\label{fig9}
\end{figure}

\begin{description}
\item [Row $1$] There are three ad hoc values plus the elements of the $\texttt{L}$,
$\texttt{D}$, $\texttt{C}$, $\texttt{E}$, $\texttt{B}$ and $\texttt{A}$ diagonals. Namely:
$$\E(\overline{R}_1)=( n-1, 5n, 9n+2, -(8n+2) , n+2, -2n , -(7n+1), 6n+1, -(5n+1) )$$
and
$$\S(\overline{R}_1)=( n-1, 6n-1, 15n+1, 7n-1, 8n+1, 6n+1, -n, 5n+1, 0 ).$$

\item [Row $2$] It is not hard to see that from the
 $\texttt{K}$,
$\texttt{G}$, $\texttt{N}$, $\texttt{D}$, $\texttt{C}$, $\texttt{I}$, $\texttt{E}$, $\texttt{B}$ and $\texttt{A}$
diagonal cells we get:
$$\E(\overline{R}_2)=( -(3n+4), -(n-2), 4n+1, 9n+3, -(8n+3) , 2n-1, -(2n-2) , -(7n+2), 6n+2 )$$
and
$$\S(\overline{R}_2)=( -(3n+4), -(4n+2), -1, 9n+2, n-1, 3n-2, n, -(6n+2), 0 ).$$

\item [Row $3$] It is not hard to see that from the
 $\texttt{A}$,
$\texttt{M}$, $\texttt{G}$, $\texttt{L}$, $\texttt{D}$, $\texttt{C}$, $\texttt{I}$, $\texttt{E}$ and $\texttt{B}$
diagonal cells we get:
$$\E(\overline{R}_3)=( 5n+3, -(4n+3), -(n-3), 5n-1, 9n+4, -(7n+4), 2n-3, -(2n-4) , -(7n+3) )$$
and
$$\S(\overline{R}_3)=(  5n+3, n, 3, 5n+2, 14n+6, 7n+2, 9n-1, 7n+3, 0 ).$$

\item [Row $4$ to $\frac{n-1}{2}$] We have to distinguish two cases, depending on the parity of the row $r$.
If $r$ is even, then write $r=4+2i$ where $ i\in \left[0,\frac{n-11}{4}\right]$. It is not hard to see that from the
 $\texttt{B}$,
$\texttt{A}$, $\texttt{K}$, $\texttt{G}$, $\texttt{N}$, $\texttt{D}$, $\texttt{C}$, $\texttt{I}$ and $\texttt{E}$
diagonal cells we get:
$$\E(\overline{R}_{4+2i})=\left(
 -(6n+4+2i), 5n+4+2i, -(3n+5+i), -(n-4-2i), 4n-i, 8n+5+2i,\right.$$
 $$ \left.  -(7n+5+2i), 2n-5-4i, -(2n-6-4i) \right)$$
 and
 $$\S(\overline{R}_{4+2i})=( -(6n+4+2i), -n, -(4n+5+i), -(5n+1-i), -(n+1), 7n+4+2i, -1, 2n-6-4i, 0 ). $$
If $r$ is odd, then write $r=5+2i$, where $i\in \left[0,\frac{n-11}{4}\right]$.
It is not hard to see that from the
 $\texttt{B}$,
$\texttt{A}$, $\texttt{M}$, $\texttt{G}$, $\texttt{L}$, $\texttt{D}$, $\texttt{C}$, $\texttt{I}$ and $\texttt{E}$
diagonal cells we get:
$$\E(\overline{R}_{5+2i})=\left(
 -(6n+5+2i), 5n+5+ 2i, -(4n+4+i), -(n-5-2i), 5n-2-i,\right.$$
$$ \left.  8n+6+2i, -(7n+6+2i)
, 2n-7-4i, -(2n-8-4i) \right)$$
and
$$\S(\overline{R}_{5+2i})=(-(6n+5+2i), -n, -(5n+4+i), -(6n-1-i), -(n+1), 7n+5+2i, -1, 2n-8-4i, 0).$$

\item [Row $\frac{n+1}{2}$] There are three ad hoc values plus the elements of the $\texttt{B}$,
$\texttt{A}$, $\texttt{K}$, $\texttt{O}$, $\texttt{D}$ and $\texttt{C}$ diagonals. Namely:
$$\E\left(\overline{R}_{\frac{n+1}{2}}\right)=\left( -3n , -\frac{13n+1}{2}, \frac{11n+1}{2},
-\frac{13n+13}{4}, n, \frac{17n+9}{4}, \frac{17n+3}{2},  -\frac{15n+3}{2} , n+1 \right)$$
and
$$\S\left(\overline{R}_{\frac{n+1}{2}}\right)=\left( -3n, -\frac{19n+1}{2}, -4n, -\frac{29n+13}{4},
-\frac{25n+13}{4}, -(2n+1), \frac{13n+1}{2}, \right.$$
$$\left. -(n+1), 0 \right).$$

\item [Row $\frac{n+3}{2}$ to $n-2$] We have to distinguish two cases, depending on the parity of the row $r$.
If $r$ is odd, then write $r=\frac{n+3}{2}+2i$ where  $i\in \left[0,\frac{n-7}{4}\right]$.
It is not hard to see that from the
 $\texttt{F}$,
$\texttt{H}$, $\texttt{B}$, $\texttt{A}$, $\texttt{P}$, $\texttt{J}$, $\texttt{Q}$, $\texttt{D}$ and $\texttt{C}$
diagonal cells we get:
$$\E\left(\overline{R}_{\frac{n+3}{2}+2i}\right)=\left(
 2n+2+4i, -(2n+3+4i), -\frac{13n+3}{2}-2i, \frac{11n+3}{2}+2i, -\frac{15n+7}{4}+i, \right.$$
$$\left.  \frac{n-3}{2}-2i, \frac{13n+17}{4}+i, \frac{17n+5}{2}+2i, -\frac{15n+5}{2}-2i \right)$$
  and
  $$\S\left(\overline{R}_{\frac{n+3}{2}+2i}\right)=\left(  2n+2+4i, -1, -\frac{13n+5}{2}-2i, -(n+1), -\frac{19n+11}{4}+i,
-\frac{17n+17}{4}-i, \right.$$
$$\left. -n, \frac{15n+5}{2}+2i, 0 \right).$$
If $r$ is even, then write  $r=\frac{n+5}{2}+2i$ where $i\in \left[0,\frac{n-11}{4}\right]$.
  It is not hard to see that from the
 $\texttt{F}$,
$\texttt{H}$, $\texttt{B}$, $\texttt{A}$, $\texttt{R}$, $\texttt{J}$, $\texttt{O}$, $\texttt{D}$ and $\texttt{C}$
diagonal cells we get:
  $$\E\left(\overline{R}_{\frac{n+5}{2}+2i}\right)=\left(
2n+4+4i, -(2n+5+4i) , -\frac{13n+5}{2}-2i, \frac{11n+5}{2}+2i,
  -\frac{19n-1}{4}+i, \right.$$
$$\left.   \frac{n-5}{2}-2i, \frac{17n+13}{4}+i, \frac{17n+7}{2}+2i, -\frac{15n+7}{2}-2i \right)$$
  and
    $$\S\left(\overline{R}_{\frac{n+5}{2}+2i}\right)=\left( 2n+4+4i, -1, -\frac{13n+7}{2}-2i, -(n+1), -\frac{23n+3}{4}+i,
-\frac{21n+13}{4}-i, \right.$$
$$ \left. -n, \frac{15n+7}{2}+2i, 0 \right).$$

\item [Row $n-1$] There are two ad hoc values plus the elements of the $\texttt{D}$,
$\texttt{C}$, $\texttt{F}$, $\texttt{H}$, $\texttt{B}$, \texttt{A} and $\texttt{R}$ diagonals. Namely:
$$\E(\overline{R}_{n-1})=\left( 9n, -8n, 3n-3, -(3n-2) , -(7n-1), 6n-1, -\frac{9n+3}{2}, -\frac{n-1}{2}, 5n+2 \right)$$
and
$$\S(\overline{R}_{n-1})=\left( 9n, n, 4n-3, n-1, -6n, -1, -\frac{9n+5}{2}, -(5n+2), 0 \right).$$

\item [Row $n$] There are four ad hoc values plus the elements of the $\texttt{D}$,
$\texttt{C}$, $\texttt{F}$, $\texttt{B}$ and $\texttt{A}$ diagonals. Namely:
$$\E(\overline{R}_n)=(3n+3, 9n+1, -(8n+1) , 3n-1, -(3n+1), -7n, 6n, -(3n+2), 1 )$$
and
$$\S(\overline{R}_{n})=( 3n+3, 12n+4, 4n+3, 7n+2, 4n+1, -(3n-1), 3n+1, -1, 0 ).$$
\end{description}
Now we list the elements and the partial sums of the columns.
\begin{description}
\item[Column $1$] There are three ad hoc values plus the elements of the $\texttt{K}$,
$\texttt{A}$, $\texttt{B}$, $\texttt{F}$, $\texttt{C}$ and $\texttt{D}$ diagonals. Namely:
$$\E(\overline{C}_1)=( n-1, -(3n+4), 5n+3, -(6n+4), -3n, 2n+2, -(8n-1), 9n, 3n+3 )$$
and
$$\S(\overline{C}_1)=( n-1, -(2n+5), 3n-2, -(3n+6), -(6n+6), -(4n+4), -(12n+3), -(3n+3), 0 ).$$

\item[Column $2$]  It is not hard to see that from the
 $\texttt{L}$,
$\texttt{G}$, $\texttt{M}$, $\texttt{A}$, $\texttt{B}$, $\texttt{H}$, $\texttt{F}$, $\texttt{C}$ and $\texttt{D}$
diagonal cells we get:
$$\E(\overline{C}_2)=( 5n, -(n-2), -(4n+3), 5n+4, -(6n+5), -(2n+3), 2n+4, -8n, 9n+1 )$$
and
$$\S(\overline{C}_2)=( 5n, 4n+2, -1, 5n+3, -(n+2), -(3n+5), -(n+1), -(9n+1), 0).$$

\item[Column $3$]  It is not hard to see that from the
 $\texttt{D}$,
$\texttt{N}$, $\texttt{G}$, $\texttt{K}$, $\texttt{A}$, $\texttt{B}$, $\texttt{H}$, $\texttt{F}$ and $\texttt{C}$
diagonal cells we get:
$$\E(\overline{C}_3)=( 9n+2, 4n+1, -(n-3), -(3n+5), 5n+5, -(6n+6), -(2n+5), 2n+6, -(8n+1))$$
and
$$\S(\overline{C}_3)=( 9n+2, 13n+3, 12n+6, 9n+1, 14n+6, 8n, 6n-5, 8n+1, 0 ).$$

\item[Column $4$]  It is not hard to see that from the
 $\texttt{C}$,
$\texttt{D}$, $\texttt{L}$, $\texttt{G}$, $\texttt{M}$, $\texttt{A}$, $\texttt{B}$, $\texttt{H}$ and $\texttt{F}$
diagonal cells we get:
$$\E(\overline{C}_4)=( -(8n+2), 9n+3, 5n-1, -(n-4), -(4n+4), 5n+6, -(6n+7), -(2n+7), 2n+8)$$
and
$$\S(\overline{C}_4)=( -(8n+2), n+1, 6n, 5n+4, n, 6n+6, -1, -(2n+8), 0 ).$$

\item[Column $5$]  It is not hard to see that from the
 $\texttt{C}$,
$\texttt{D}$, $\texttt{N}$, $\texttt{G}$, $\texttt{K}$, $\texttt{A}$, $\texttt{B}$, $\texttt{H}$ and $\texttt{F}$
diagonal cells we get:
$$\E(\overline{C}_5)=( -(8n+3), 9n+4, 4n, -(n-5), -(3n+6), 5n+7, -(6n+8), -(2n+9), 2n+10 )$$
and
$$\S(\overline{C}_5)=( -(8n+3), n+1, 5n+1, 4n+6, n, 6n+7, -1, -(2n+10), 0 ).$$

\item[Column $6$ to $\frac{n-1}{2}$]  We have to distinguish two cases, depending on the parity of the column $c$.
If $c$ is even, then write $c=6+2i$ where $i \in [0,\frac{n-15}{4}]$.
 It is not hard to see that from the
 $\texttt{C}$,
$\texttt{D}$, $\texttt{L}$, $\texttt{G}$, $\texttt{M}$, $\texttt{A}$, $\texttt{B}$, $\texttt{H}$ and $\texttt{F}$
diagonal cells we get:
$$\E(\overline{C}_{6+2i})=\left( -(7n+4+2i), 8n+5+2i, 5n-2-i, -(n-6-2i), -(4n+5+i),  \right.$$
$$\left. 5n+8+2i,  -(6n+9+2i), -(2n+11+4i), 2n+12+4i \right)$$
  and
$$\S(\overline{C}_{6+2i})=( -(7n+4+2i), n+1, 6n-1-i, 5n+5+i, n, 6n+8+2i, -1, -(2n+12+4i), 0 ).$$
If $c$ is odd, then write $c=7+2i$ where $i \in [0,\frac{n-15}{4}]$.
It is not hard to see that from the
 $\texttt{C}$,
$\texttt{D}$, $\texttt{N}$, $\texttt{G}$, $\texttt{K}$, $\texttt{A}$, $\texttt{B}$, $\texttt{H}$ and $\texttt{F}$
diagonal cells we get:
$$\E(\overline{C}_{7+2i})=\left( -(7n+5+2i), 8n+6+2i, 4n-1-i, -(n-7-2i), -(3n+7+i),  \right.$$
 $$\left. 5n+9+2i, -(6n+10+2i), -(2n+13+4i), 2n+14+4i \right)$$
  and
$$\S(\overline{C}_{7+2i})=(-(7n+5+2i), n+1, 5n-i, 4n+7+i, n, 6n+9+2i, -1, -(2n+14+4i), 0).$$

\item[Column $\frac{n+1}{2}$] The are three ad hoc values plus the elements of the $\texttt{C}$,
$\texttt{D}$, $\texttt{L}$, $\texttt{P}$, $\texttt{A}$ and $\texttt{B}$ diagonals. Namely:
$$\E\left(\overline{C}_{\frac{n+1}{2}}\right)=\left( n+2, -\frac{15n-3}{2}, \frac{17n-1}{2}, \frac{19n+3}{4} , n, -\frac{15n+7}{4}, \frac{11n+5}{2},
  -\frac{13n+7}{4},\right.$$
  $$\left. -(3n+1) \right)$$
  and
$$\S\left(\overline{C}_{\frac{n+1}{2}}\right)=\left(n+2, -\frac{13n-7}{2}, 2n+3, \frac{27n+15}{4}, \frac{31n+1}{4}, 4n+2, \frac{19n+9}{2},
  3n+1, 0 \right).$$

\item[Column $\frac{n+3}{2}$ to $n-2$]  We have to distinguish two cases, depending on the parity of the
column $c$. If $c$ is odd, then write $c=\frac{n+3}{2}+2i$ where $i \in [0,\frac{n-7}{4}]$.
It is not hard to see that from the
 $\texttt{E}$,
$\texttt{I}$, $\texttt{C}$, $\texttt{D}$, $\texttt{O}$, $\texttt{J}$, $\texttt{R}$, $\texttt{A}$ and $\texttt{B}$
diagonal cells we get:
$$\E\left(\overline{C}_{\frac{n+3}{2}+2i}\right)=\left(-(2n-4i), 2n-1-4i, -\frac{15n-1}{2}-2i, \frac{17n+1}{2}+2i, \frac{17n+9}{4}+i, \right.$$
$$\left. \frac{n-3}{2}-2i,  -\frac{19n-1}{4}+i, \frac{11n+7}{2}+2i,
  -\frac{13n+9}{2}-2i,
\right)$$
and
$$\S\left(\overline{C}_{\frac{n+3}{2}+2i}\right)=\left(-(2n-4i), -1, -\frac{15n+1}{2}-2i, n, \frac{21n+9}{4}+i, \right.$$
$$ \left. \frac{23n+3}{4}-i, n+1, \frac{13n+9}{2}+2i, 0 \right).$$
If $c$ is even, then write $c=\frac{n+5}{2}+2i$ where $i \in [0,\frac{n-11}{4}]$.
It is not hard to see that from the
 $\texttt{E}$,
$\texttt{I}$, $\texttt{C}$, $\texttt{D}$, $\texttt{Q}$, $\texttt{J}$, $\texttt{P}$, $\texttt{A}$ and $\texttt{B}$
diagonal cells we get:
  $$\E\left(\overline{C}_{\frac{n+5}{2}+2i}\right)=\left( -(2n-2-4i), 2n-3-4i, -\frac{15n+1}{2}-2i, \frac{17n+3}{2}+2i, \frac{13n+17}{4}+i,\right.$$
 $$\left. \frac{n-5}{2}-2i, -\frac{15n+3}{4}+i, \frac{11n+9}{2}+2i, -\frac{13n+11}{2}-2i \right)$$
  and
$$\S\left(\overline{C}_{\frac{n+5}{2}+2i}\right)=\left( -(2n-2-4i), -1, -\frac{15n+3}{2}, n, \frac{17n+17}{4}+i, \frac{19n+7}{4}-i, n+1,\right.$$
$$\left.  \frac{13n+11}{2}+2i, 0 \right).$$

\item[Column $n-1$] There are two ad hoc values plus the elements of the $\texttt{A}$,
$\texttt{B}$, $\texttt{E}$, $\texttt{I}$, $\texttt{C}$, $\texttt{D}$ and $\texttt{Q}$ diagonals. Namely:
$$\E(\overline{C}_{n-1})=\left(6n+1, -(7n+2), -(n+5), n+4, -(8n-3), 9n-2, \frac{7n+5}{2},\right.$$
$$\left. -\frac{n-1}{2}, -(3n+2) \right)$$
and
$$\S(\overline{C}_{n-1})=\left( 6n+1, -(n+1), -(2n+6), -(n+2), -(9n-1), -1, \frac{7n+3}{2}, 3n+2, 0 \right).$$

\item[Column $n$] There are four ad hoc values plus the elements of the $\texttt{A}$,
$\texttt{B}$, $\texttt{E}$,  $\texttt{C}$ and $\texttt{D}$ diagonals. Namely:
$$\E(\overline{C}_n)=(-(5n+1), 6n+2, -(7n+3), -(n+3), n+1, -(8n-2), 9n-1, 5n+2, 1 )$$
and
$$\S(\overline{C}_n)=(-(5n+1), n+1, -(6n+2), -(7n+5), -(6n+4), -(14n+2), -(5n+3), -1, 0 ).$$
\end{description}

Finally, we consider the support of $A$:
$$\begin{array}{rcl}
supp(A)& = & \{1\} \cup \left[2,\frac{n-3}{2}\right]_{(\texttt{J})}\cup \{\frac{n-1}{2} \} \cup
\left[\frac{n+1}{2},n-2\right]_{(\texttt{G})}\cup \{n-1,n,n+1,n+2\}  \cup\\[3pt]
&&\cup \left[n+3,2n\right]_{(\texttt{E}\cup \texttt{I})}\cup
\left[2n+2,3n-1\right]_{(\texttt{F}\cup \texttt{H})}\cup  \{3n,3n+1,3n+2,3n+3\}
\cup\\ [3pt]
&& \cup \left[3n+4,\frac{13n+13}{4}\right]_{(\texttt{K})}
\cup\left[\frac{13n+17}{4},\frac{7n+5}{2}\right]_{(\texttt{Q})}   \cup
\left[\frac{7n+7}{2},\frac{15n+7}{4}\right]_{(\texttt{P})} \cup \\[3pt]
&& \cup \left[\frac{15n+11}{4},4n+1\right]_{(\texttt{N})}
\cup  \left[4n+3,\frac{17n+5}{4}\right]_{(\texttt{M})}
\cup \left[\frac{17n+9}{4},\frac{9n+1}{2}\right]_{(\texttt{O})}\cup\\ [3pt]
&& \cup \left[\frac{9n+3}{2},\frac{19n-1}{4}\right]_{(\texttt{R})} \cup \left[\frac{19n+3}{4},5n\right]_{(\texttt{L})}
\cup \{5n+1,5n+2\}  \cup \left[5n+3,6n+2\right]_{(\texttt{A})}
\cup \\ [3pt]
&& \cup \left[6n+4,7n+3\right]_{(\texttt{B})} \cup \left[7n+4,8n+3\right]_{(\texttt{C})}\cup
\left[8n+5,9n+4\right]_{(\texttt{D})}\\ [3pt]
& =& [1,9n+4]\setminus\{2n+1,4n+2,6n+3,8n+4\}.
\end{array}$$
This concludes the proof.
\end{proof}

\begin{ex}\label{4.9}
Following the proof of Proposition \ref{prop:9} we obtain the
integer globally simple $\H_9(15; 9)$ below.
 \begin{center}
\begin{Tiny}
$\begin{array}{|r|r|r|r|r|r|r|r|r|r|r|r|r|r|r|}\hline
     14&    75&   137&  -122&    &    &    &    17&   -30&    &    &    &  -106&    91&   -76 \\\hline
     -49&   -13&    61&   138&  -123&    &    &    &    29&   -28&    &    &    &  -107&    92 \\\hline
      78&   -63&   -12&    74&   139&  -109&    &    &    &    27&   -26&    &    &    &  -108 \\\hline
     -94&    79&   -50&   -11&    60&   125&  -110&    &    &    &    25&   -24&    &    &     \\\hline
      &   -95&    80&   -64&   -10&    73&   126&  -111&    &    &    &    23&   -22&    &     \\\hline
      &    &   -96&    81&   -51&    -9&    59&   127&  -112&    &    &    &    21&   -20&      \\\hline
      &    &    &   -97&    82&   -65&    -8&    72&   128&  -113&    &    &    &    19&   -18 \\\hline
     -45&    &    &    &   -98&    83&   -52&    15&    66&   129&  -114&    &    &    &    16 \\\hline
      32&   -33&    &    &    &   -99&    84&   -58&     6&    53&   130&  -115&    &    &      \\\hline
      &    34&   -35&    &    &    &  -100&    85&   -71&     5&    67&   131&  -116&    &      \\\hline
      &    &    36&   -37&    &    &    &  -101&    86&   -57&     4&    54&   132&  -117&      \\\hline
      &    &    &    38&   -39&    &    &    &  -102&    87&   -70&     3&    68&   133&  -118 \\\hline
    -119&    &    &    &    40&   -41&    &    &    &  -103&    88&   -56&     2&    55&   134 \\\hline
     135&  -120&    &    &    &    42&   -43&    &    &    &  -104&    89&   -69&    -7&    77 \\\hline
      48&   136&  -121&    &    &    &    44&   -46&    &    &    &  -105&    90&   -47&     1\\\hline
\end{array}$
\end{Tiny}
\end{center}
\end{ex}

\begin{lem}\label{percorso}
  For any $n\equiv 7\pmod {14}$ such that $n\geq 21$, write $r=\frac{n-7}{2}$.
   Let $A_n$ be a $9$-diagonal array whose filled diagonals are $D_1,D_2,\ldots,D_7$, $D_{r+7}$ and $D_{r+8}$.
  Then $(\R,\C)$, where $\R=(1,1,\ldots,1)$ and $\C=(\underbrace{-1,\ldots,-1}_8,1,1,\ldots,1)$,
  is a solution of $P(A_n)$.
  \end{lem}

\begin{proof}
For any $i\in[1,7]\cup\{r+7,r+8\}$ set $D_i=(d_{i,1},d_{i,2},d_{i,3},\ldots,d_{i,n})$, where $d_{i,1}$ is the position $[i,1]$ of $A_n$.
Also, we set
$$\begin{array}{rcl}
\mathbf{A}_i & =&  d_{i,8},d_{i,9},d_{i,10},\ldots,d_{i,n};\\
\mathbf{B}_i & = &  d_{1,i},d_{1,i+r},d_{1,i+2r},\ldots, d_{1,i+\frac{2r}{7}r }; \\
\mathbf{C}_i & = &  d_{r+7,i},d_{r+7,i+r},d_{r+7,i+2r},\ldots, d_{r+7,i+\frac{2r}{7}r }; \\
\mathbf{D}_{1} & = & d_{1,1}, d_{1,1+r},d_{1,1+2r},\ldots, d_{1,1+(\frac{2r}{7} - 2) r }; \\
\mathbf{D}_{2} & = & d_{1,8}, d_{1,8+r}; \\
\mathbf{E}_{1} & = & d_{r+7,1}, d_{r+7,1+r},d_{r+7,1+2r},\ldots, d_{r+7,1+(\frac{2r}{7} - 2) r }; \\
\mathbf{E}_{2} & = & d_{r+7,8},d_{r+7,8+r}.
\end{array}$$
To aid in the proof, at the webpage
\begin{center}
\texttt{http://anita-pasotti.unibs.it/Publications.html},
\end{center}
we give  a schematic picture of where each of these sequences fills cells.
By a direct check, one can verify that
$$\begin{array}{rcl}
L_{\R,\C}(d_{6,8})&=&(\mathbf{A}_{6}, d_{4,1}, d_{2,2}, d_{r+8,3}, d_{7,4}, d_{5,5}, d_{3,6}, \mathbf{B}_7,\mathbf{C}_7,
d_{6,7},\\
& & \mathbf{A}_4, d_{2,1},d_{r+8,2}, d_{7,3}, d_{5,4}, d_{3,5},\mathbf{B}_6, \mathbf{C}_6, d_{6,6}, d_{4,7},\\
& & \mathbf{A}_2, d_{r+8,1},d_{7,2}, d_{5,3}, d_{3,4}, \mathbf{B}_5, \mathbf{C}_5, d_{6,5},
d_{4,6},d_{2,7},\\
& & \mathbf{A}_{r+8}, d_{7,1},d_{5,2}, d_{3,3},\mathbf{B}_4, \mathbf{C}_4, d_{6,4},
d_{4,5},d_{2,6},d_{r+8,7},\\
& &\mathbf{A}_{7}, d_{5,1}, d_{3,2},\mathbf{B}_3, \mathbf{C}_3, d_{6,3},d_{4,4},d_{2,5},d_{r+8,6},d_{7,7},\\
 & & \mathbf{A}_{5}, d_{3,1},\mathbf{B}_2, \mathbf{C}_2, d_{6,2}, d_{4,3},d_{2,4},d_{r+8,5},d_{7,6},d_{5,7},\\
& & \mathbf{A}_{3}, \mathbf{D}_{1}, \mathbf{E}_{2}, d_{6,1},d_{4,2},d_{2,3},d_{r+8,4}, d_{7,5}, d_{5,6},
d_{3,7},\mathbf{D}_{2}, \mathbf{E}_{1}).
\end{array}$$
Hence, it is easy to see that $L_{\R,\C}(d_{6,8})$ covers all the filled cells of $A_n$.
\end{proof}

Now we are ready to prove Theorem  \ref{main}.
\begin{proof}[Proof of Theorem \ref{main}]
The result follows from Theorem \ref{thm:biembedding}, once we have proved the existence of a relative Heffter array with compatible simple orderings $\omega_r$ and $\omega_c$.

\noindent (1) For any $n$ odd, a $\H_n(n;3)$ and a $\H_{2n}(n;3)$ are constructed in Propositions \ref{prop:3-t=n} and
\ref{prop:3-t=2n}, respectively. Clearly these are globally simple Heffter arrays.
Since they are cyclically $3$-diagonal their compatibility follows from \cite[Proposition 3.4]{CMPPHeffter}.

\noindent (2) Let $n\equiv 3 \pmod 4$. A $\H_3(n;3)$ and a $\H_5(n;5)$ are constructed in \cite[Propositions 5.1 and
5.5]{CMPPRelative}, respectively.
As before these are globally simple Heffter arrays and
since they are cyclically $3$-diagonal and $5$-diagonal, respectively,
 their compatibility follows from \cite[Proposition 3.4]{CMPPHeffter}.
A globally simple $\H_7(n;7)$ is given in Proposition \ref{prop:7}.
 Since this is cyclically $7$-diagonal its compatibility follows from  \cite[Propositions 3.4 and 3.6]{CMPPHeffter}.
 Finally, a globally simple $\H_9(n;9)$ is given in Proposition \ref{prop:9}.
 Since this is $9$-diagonal with width $\frac{n-9}{2}$, if $\gcd\left(n,\frac{n-7}{2} \right)=\gcd(n,7)=1$ its
compatibility follows from  \cite[Proposition 4.19]{CDP}.
 If $\gcd(n,7)\neq 1$ the result follows from Lemma \ref{percorso}.
 \end{proof}

\section{Archdeacon arrays}\label{sec:Arch}

In this section we introduce a further generalization of the concept of Heffter array. In particular we will
consider p.f. arrays where the number of filled cells in each row and in each column is not fixed.

\begin{defi}\label{def:ArchdeaconH}
An \emph{Archdeacon array} $A$ over an abelian group $(G,+)$ is an $m\times n$  p.f.  array with elements in $G$,
such that:
\begin{itemize}
\item[($\rm{a})$] $\mathcal{E}(A)$ is a set;
\item[($\rm{b})$] for every $g\in G$, $g\in \mathcal{E}(A)$ implies $-g\not\in \mathcal{E}(A)$;
\item[($\rm{c})$] the elements in every row and column sum to $0$.
\end{itemize}
\end{defi}
An example of this kind of arrays will be given in Figure \ref{fig1}.
We note that, in the special case $G=\mathbb{Z}_v$, $\pm \E(A)=\mathbb{Z}_v-J$ where $J$ is a subgroup of $\mathbb{Z}_v$
and all the rows (resp. columns) have the same number of filled cells, we meet again the definition of a relative
Heffter array.
The purpose of this section is to show how Archdeacon arrays can be used in order to obtain biembeddings and
orthogonal cycle decompositions.
First of all we need a generalization of \cite[Proposition 2.6]{BP}, stated by Buratti in
\cite[Theorem 3.3]{B}.
All the well-known concepts about the differences method can be found in \cite{B, CMPPRelative}.

\begin{thm}\label{thm:basecycles2}
Let $G$ be an additive group and $\mathcal{B}$ be a set of cycles with vertices in $G$.
If the list of differences of $\mathcal{B}$ is a set, say $\Lambda$, then $\mathcal{B}$
is a set of base cycles of
a $G$-regular cycle decomposition of $\Cay[G:\Lambda]$.
\end{thm}

Generalizing Proposition \ref{HeffterToDecompositions}, an Archdeacon array can be used to obtain regular cycle
decompositions of Cayley graphs as follows.

\begin{prop}\label{ArchdeaconToDecompositions}
Let $A$ be an $m\times n$ Archdeacon array on an abelian group $G$ with simple orderings
$\omega_r=\omega_{\overline{R}_1}\circ \ldots \circ\omega_{\overline{R}_m}$ for the rows
and $\omega_c=\omega_{\overline{C}_1}\circ \ldots \circ\omega_{\overline{C}_n}$ for the columns. Then:
\begin{itemize}
\item[(1)] $\B_{\omega_r}=\{\S(\omega_{\overline{R}_i}) \mid i\in[1,m]\}$ is a set of base cycles of a
$G$-regular cycle decomposition $\D_{\omega_r}$ of $\Cay[G:\pm \E(A)]$;
\item[(2)] $\B_{\omega_c}=\{\S(\omega_{\overline{C}_j}) \mid j\in [1, n]\}$ is a set of base cycles of a
$G$-regular cycle decomposition $\D_{\omega_c}$ of $\Cay[G:\pm \E(A)]$;
\item[(3)] the cycle decompositions $\D_{\omega_r}$ and $\D_{\omega_c}$ are orthogonal.
\end{itemize}
\end{prop}

\begin{proof}
(1) Since the ordering $\omega_r$ is simple the elements of $\B_{\omega_r}$ are cycles of lengths
$|\E(\overline{R}_1)|,\ldots,$ $ |\E(\overline{R}_m)|$
and by definition of partial sums the list of differences
of $\S(\omega_{\overline{R}_i})$ is $\pm\E(\overline{R}_i)$, for any $i\in[1,m]$.
Hence, the list of differences  of $\B_{\omega_r}$ is $\pm \E(A)$ and so the thesis follows from Theorem
\ref{thm:basecycles2}.
Obviously, (2) can be proved in the same way. Note that, in general, the cycles of $\B_{\omega_r}$ and those of
$\B_{\omega_c}$
have different lengths.
(3) follows from the requirement that
the elements of $\pm \E(A)$ are pairwise distinct.
\end{proof}

Moreover the pair of cycles decompositions obtained from an Archdeacon array can be biembedded under the same hypothesis
of Theorem \ref{thm:biembedding}. In fact, within the same proof, we have that:
\begin{thm}\label{thm:biembedding2}
Let $A$ be an Archdeacon array on an abelian group $G$ that is simple with respect to two compatible orderings
$\omega_r$ and $\omega_c$.
Then there exists a biembedding of the $G$-regular cycle decompositions $\mathcal{D}_{\omega_r^{-1}}$ and
$\mathcal{D}_{\omega_c}$ of  $\Cay[G:\pm \E(A)]$ into an orientable surface.
\end{thm}

We observe that if an Archdeacon array has no empty rows/columns, then
a necessary condition for the existence of compatible orderings is  $|skel(A)|\equiv m+n-1 \pmod{2}$.
This can be proved with the same proof of \cite[Theorem 1.1]{CDDYbiem} and of \cite[Theorem 2.7]{CDP}.

Finally, as an easy consequence of Theorem \ref{thm:biembedding2}, we obtain the relationship between the Crazy Knight's
Tour Problem and globally simple Archdeacon arrays.

\begin{cor}\label{FromArchdeaconToKN}
  Let $A$ be a globally simple Archdeacon array on an abelian group $G$ such that $P(A)$ admits a solution $(\mathcal{R},\mathcal{C})$.
  Then there exists a biembedding of the $G$-regular cycle decompositions $\mathcal{D}_{\omega_r^{-1}}$ and
$\mathcal{D}_{\omega_c}$ of $\Cay[G:\pm \E(A)]$ into an orientable surface.
\end{cor}

Given two $m\times n$ p.f. arrays $A$ and $B$ defined on abelian groups $G_1$ and $G_2$, respectively, we
define their direct sum $A\oplus B$ as the $m\times n$ p.f. array $E$ whose skeleton is  $skel(A)\cup skel(B)$ and whose entries in $G_1\oplus G_2$ are so defined:
$$E[i,j]=\left\{\begin{array}{cl}
(A[i,j],B[i,j]) & \mbox{ if } (i,j)\in skel(A)\cap skel(B),\\
(A[i,j],0_{G_2}) &\mbox{ if } (i,j)\in skel(A)\setminus skel(B),\\
(0_{G_1},B[i,j]) &\mbox{ if } (i,j)\in skel(B)\setminus skel(A).
\end{array}\right.$$

In the following we will denote by
$\overline{R}_i(A)$ and $\overline{C}_j(A)$ the $i$-th row and the $j$-th column of $A$, respectively.

\begin{lem}\label{GlobSimpleness}
Let $A$ and $B$ be $m\times n$ globally simple p.f. arrays over
abelian groups  $G_1$ and $G_2$, respectively, such that:
\begin{itemize}
\item[(1)] for any $i\in [1,m]$ for which the $i$-th rows of $A$ and $B$ are both
nonempty, we have $skel(\overline{R}_i(A)) \cap skel(\overline{R}_i(B)) \not=\emptyset$;
\item[(2)] for any $j\in [1,n]$ for which the $j$-th columns of $A$ and $B$ are both
nonempty, we have $skel(\overline{C}_j(A)) \cap skel(\overline{C}_j(B)) \not=\emptyset$;
\item[(3)] the elements in every nonempty row/column of both $A$ and $B$ sum to zero.
\end{itemize}
Then $A\oplus B$ is a globally simple p.f. array, whose nonempty rows and columns sum to zero.
\end{lem}

\begin{proof}
Since the elements in every nonempty row and column of  both $A$ and $B$  sum to zero, the same holds for $A\oplus B$.

Let us suppose, by contradiction, that there exists a row (resp. a column) $\overline{R}_i$ of $A\oplus B$ that is not simple with respect to the  natural ordering.
Then there would be a subsequence $L$ of consecutive elements of $\overline{R}_i$ that sum to zero.
Denoted by $L_1$ the subsequence of the first coordinates of $L$ (ignoring the zeros) and by $L_2$ the one of the second coordinates,
we have that both $L_1$ and $L_2$ sums to zero.
Since both $\overline{R}_i(A)$ and $\overline{R}_i(B)$ are simple with respect to the natural ordering, it follows
that either $L_1=\emptyset$ (we are ignoring zeros) or  $L_1=\E(A)$. Similarly, for $\overline{R}_i(B)$.
If $L_1=\emptyset$, then $L_2=\E(\overline{R}_i(B))$ and hence $L$ is $\E(\overline{R}_i)$.
Similarly, if $L_2=\emptyset$.
Finally, if $L_1$ and $L_2$ are both nonempty,
the only possibility is that $L=\E(\overline{R}_i)$
since  $skel(\overline{R}_i(A)) \cap skel(\overline{R}_i(B)) \not=\emptyset$.
\end{proof}

\begin{prop}\label{ArchdeaconAlmost}
Let $A$ be an Archdeacon array over an abelian group $G_1$ and 
let $B$ be a p.f. array of the same size defined over an abelian group $G_2$.
Suppose that the hypotheses of {\rm Lemma \ref{GlobSimpleness}} are satisfied, 
that $\E(A\oplus B)$ is a set and that
if $(0_{G_1},x) \in \E(A\oplus B)$, then $(0_{G_1},-x)\not \in \E(A\oplus B)$.
Then $A \oplus B$ is a globally simple Archdeacon array over $G_1\oplus G_2$.
\end{prop}

\begin{proof}
By Lemma \ref{GlobSimpleness}, $E=A\oplus B$ is a globally simple p.f. array whose rows and columns sum to zero.
We now show that condition (b) of Definition \ref{def:ArchdeaconH} holds.
Suppose that $g=(g_1,g_2)\in G_1\oplus G_2$ belongs to $\E(E)$.
Then, either $g_1\in \E(A)$ or $g_1=0_{G_1}$. In the first case, $-g_1 \not \in \E(A)$ and so $-g=(-g_1,-g_2)\not \in 
\E(E)$.
If $g_1=0_{G_1}$, then $(0_{G_1},-g_2)\not \in \E(E)$ by hypothesis, proving the statement.
\end{proof}

Now we consider the $m\times n$ p.f. array $B_{m,n,d}(i_1,i_2;j_1,j_2)$ over $\mathbb{Z}_d$ which has only four nonempty cells:
those in positions  $(i_1,j_1),(i_2,j_2)$ that we fill with $+1$ and
those  in positions $(i_2,j_1),(i_1,j_2)$ that we fill with $-1$.
The following result is a consequence of Proposition \ref{ArchdeaconAlmost}.

\begin{cor}\label{ultimo}
Let $k<n$ and let us suppose there exists a globally simple  cyclically $k$-diagonal  $\H_t(n;k)$, say $A$, whose
filled diagonals are
$D_1,\dots,D_k$. Then considering the array
$B=B_{n,n,d}(1,2;1,2)$, where $d>2$,
we have that $E=A \oplus B$ is a globally simple
Archdeacon array over the group $\mathbb{Z}_{2nk+t}\oplus \mathbb{Z}_d$.
\end{cor}

We know that there exists a (globally simple) cyclically $3$-diagonal $\H_t(n;3)$ in each of the following cases:
\begin{itemize}
\item[(1)] $t\in \{1,2\}$ and $n\equiv 0,1 \pmod{4}$, see \cite[Theorems 3.4 and 3.9]{ADDY};
\item[(2)] $t=3$ and $n\equiv 0,3\pmod{4}$, see \cite[Propositions 5.1 and 5.3]{CMPPRelative};
\item[(3)] $t=n$ and $n$ is odd, see Proposition \ref{prop:3-t=n};
\item[(4)] $t=2n$ and $n$ is odd, see Proposition \ref{prop:3-t=2n}.
\end{itemize}
Therefore in these cases, we can apply Corollary \ref{ultimo}:
for any $d\geq 3$  there exists a globally simple Archdeacon array $E$ of size $n\geq 4$ defined over
$\mathbb{Z}_{6n+t}\oplus \mathbb{Z}_d$ whose skeleton is $D_1\cup D_2 \cup D_3 \cup \{(1,2)\}$.

Moreover, because of \cite[Proposition 5.9]{CDP}, there exists a solution of $P(E)$ whenever $n$ is also even. In
those cases we have a biembedding of $\Cay[\mathbb{Z}_{6n+t}\oplus \mathbb{Z}_d:\pm \mathcal{E}(E)]$ in an orientable
surface whose faces classes contain triangles and exactly one quadrangle.

As example of such construction, in Figure \ref{fig1} we give a globally simple Archdeacon array over $\Z_{51}\oplus \Z_d$, where $d\geq 3$.

\begin{figure}[ht]
\begin{footnotesize}
$$\begin{array}{|c|c|c|c|c|c|c|c|}\hline
 (-9,1) &  (0,-1)  &     &     &     &     &    (16,0) &   (-7,0) \\\hline
(-3,-1) &  (-22,1) &     &     &     &     &     &    (25,0) \\\hline
     (12,0) &    (1,0)&   (-13,0) &    &     &     &     &      \\\hline
      &    (21,0) &    (2,0) &  (-23,0)&    &     &     &      \\\hline
      &     &    (11,0) &    (8,0) &  (-19,0) &    &     &      \\\hline
      &     &     &    (15,0) &    (5,0) &  (-20,0) &    &      \\\hline
      &     &     &     &    (14,0) &   (-4,0) &  (-10,0) &     \\\hline
      &     &     &     &     &    (24,0) &   (-6,0) &  (-18,0)  \\\hline
\end{array}$$
\end{footnotesize}
\caption{An  Archdeacon array over $\Z_{51}\oplus \Z_d$.}\label{fig1}
\end{figure}

We recall that the existence of a (globally simple) cyclically $4$-diagonal $\H_t(n;4)$
for any $n$ and $t\in \{1,2,4\}$ has been proved in \cite[Theorem 2.2]{DW} and  \cite[Proposition 4.9]{CMPPRelative}.
Therefore, for any $d\geq 3$, because of Corollary \ref{ultimo} there exists a globally simple Archdeacon array $E$ of
size $n\geq 4$ over  $\mathbb{Z}_{8n+t}\oplus \mathbb{Z}_d$ whose skeleton is $D_1\cup D_2 \cup D_3 \cup D_4
\cup\{(1,2)\}$.

Moreover, because of \cite[Proposition 5.13]{CDP}, there exists a solution of $P(E)$ whenever $n\not\equiv 0
\pmod{3}$. In these cases we have a biembedding of $\Cay[\mathbb{Z}_{8n+t}\oplus \mathbb{Z}_d:\pm \mathcal{E}(E)]$ in
an orientable surface whose faces classes contain quadrangles and exactly one pentagon.

An  example of such construction is given in Figure \ref{fig1_2} where we provide a globally simple Archdeacon array over $\Z_{60}\oplus \Z_d$, where $d\geq 3$.

\begin{figure}[ht]
\begin{footnotesize}
$$\begin{array}{|c|c|c|c|c|c|c|}
\hline  (25,1) &(0,-1)  && & (1,0) & (-8,0) & (-18,0)      \\
\hline   (-19,-1) & (26,1) & &&  & (2,0) & (-9,0)   \\
\hline   (-10,0) & (-20,0)  & (27,0) & & & & (3,0)  \\
\hline    (4 ,0)& (-11,0) & (-21,0) & (28,0) &   &  &\\
\hline     &   (5,0) & (-12,0) & (-22,0)&(29,0) &  &  \\
\hline    &  & (6,0) & (-13,0)&(-16,0)  & (23,0) &  \\
\hline   &  &  & (7,0) &(-14,0)  & (-17,0) & (24,0) \\
\hline
\end{array}$$
\end{footnotesize}
\caption{An Archdeacon array over $\Z_{60}\oplus \Z_d$.}\label{fig1_2}
\end{figure}


\begin{thebibliography}{50}

\bibitem{AL} B. Alspach \and G. Liversidge,
On strongly sequenceable abelian groups,
\textit{Art Discrete Appl. Math.}, to appear.

\bibitem{A} D.S. Archdeacon,
Heffter arrays and biembedding graphs on surfaces,
\textit{Electron. J. Combin.} 22 (2015) \#P1.74.

\bibitem{ABD} D.S. Archdeacon, T. Boothby \and J.H. Dinitz,
Tight Heffter arrays exist for all possible values,
\textit{J. Combin. Des.} 25 (2017), 5--35.

\bibitem{ADDY} D.S. Archdeacon, J.H. Dinitz, D.M. Donovan \and E.\c{S}. Yaz\i c\i,
Square integer Heffter arrays with empty cells,
\textit{Des. Codes Cryptogr.} 77 (2015), 409--426.

\bibitem{ADMS} D.S. Archdeacon, J.H. Dinitz, A. Mattern \and D.R. Stinson,
On partial sums in cyclic groups,
\textit{J. Combin. Math. Combin. Comput.} 98 (2016), 327--342.

\bibitem{B} M. Buratti,
Cycle decompositions with a sharply vertex transitive automorphism group,
\textit{Le Matematiche} VOL. LIX (2004), 91--105.

\bibitem{BP} M. Buratti \and A. Pasotti,
Graph decompositions with the use of difference matrices,
\textit{Bull. Inst. Combin. Appl.} 47 (2006), 23--32.

\bibitem{BP2007} M. Buratti \and A. Pasotti,
On perfect $\Gamma$-decompositions of the complete graph,
\textit{J. Combin. Des.} 17 (2008), 197--209.

\bibitem{BCDY} K. Burrage, D.M. Donovan, N.J. Cavenagh,  \and  E.\c{S}. Yaz\i c\i,
Globally simple Heffter arrays $H(n;k)$ when $k\equiv 0,3 \pmod{4}$,
\textit{Discrete Math.} 343 (2020), \#111787.

\bibitem{CDDY} N.J. Cavenagh, J. Dinitz, D.M. Donovan \and  E.\c{S}. Yaz\i c\i,
The existence of square non-integer Heffter arrays, 
\textit{Ars Math. Contemp.} 17 (2019), 369--395. 

\bibitem{CDDYbiem} N.J. Cavenagh, D. Donovan \and  E.\c{S}. Yaz\i c\i,
Biembeddings of cycle systems using integer Heffter arrays,
preprint available at \texttt{https://arxiv.org/abs/1906.10525}.

\bibitem{CDP} S. Costa, M. Dalai \and A. Pasotti,
A tour problem on a toroidal board,
\textit{Austral. J. Combin.}, 76 (2020), 183--207.

\bibitem{CMPPSums} S. Costa, F. Morini, A. Pasotti \and  M.A. Pellegrini,
A problem on partial sums in abelian groups,
\textit{Discrete Math.} 341 (2018), 705--712.

\bibitem{CMPPHeffter} S. Costa, F. Morini, A. Pasotti \and M.A. Pellegrini,
Globally simple Heffter arrays and orthogonal cyclic cycle decompositions,
\textit{Austral. J. Combin.} 72 (2018), 549--493.

\bibitem{CMPPRelative} S. Costa, F. Morini, A. Pasotti \and M.A. Pellegrini,
A generalization of Heffter arrays,
\textit{J. Combin. Des.} 28 (2020), 171--206.

\bibitem{DM} J.H. Dinitz \and A.R.W. Mattern,
Biembedding Steiner triple systems and $n$-cycle systems on orientable surfaces,
\textit{Austral. J. Combin.} 67 (2017), 327--344.

\bibitem{DW} J.H. Dinitz \and I.M. Wanless,
The existence of square integer Heffter arrays,
\textit{Ars Math. Contemp.} 13 (2017), 81--93.

\bibitem{GT} J.L. Gross \and T.W. Tucker,
\textit{Topological Graph Theory},
John Wiley, New York, 1987.

\bibitem{HOS} J. Hicks,  M.A. Ollis \and J.R. Schmitt,
Distinct partial sums in cyclic groups: polynomial method and constructive approaches,
\textit{J. Combin. Des.} 27 (2019), 369--385.

\bibitem{Moh} B. Mohar,
Combinatorial local planarity and the width of graph embeddings,
\emph{Canad. J. Math.} 44 (1992), 1272--1288.

\bibitem{MT} B. Mohar \and C. Thomassen,
\textit{Graphs on surfaces},
Johns Hopkins University Press, Baltimore, 2001.

\bibitem{MP} F. Morini \and M.A. Pellegrini,
On the existence of integer relative Heffter arrays,
preprint available at \texttt{https://arxiv.org/abs/1910.09921}.


\bibitem{O} M.A. Ollis,
Sequences in dihedral groups with distinct partial products,
preprint available at \texttt{https://arxiv.org/abs/1904.07646}.

\end{thebibliography}
\end{document}